\documentclass{amsart}
\usepackage{amsmath,amsthm,amssymb}
\usepackage{hyperref}
\usepackage{pstricks,pst-node,multido,pst-coil}

\usepackage{a4wide}

\title{Moments of Askey-Wilson polynomials}

\author{Jang Soo Kim}
\address{School of Mathematics, University of Minnesota, Minneapolis, MN 55455}
\email{kimjs@math.umn.edu}

\author{Dennis Stanton}
\address{School of Mathematics, University of Minnesota, Minneapolis, MN 55455}
\email{stanton@math.umn.edu}

\date{\today}


\newtheorem{thm}{Theorem}[section]
\newtheorem{lem}[thm]{Lemma}
\newtheorem{prop}[thm]{Proposition}
\newtheorem{cor}[thm]{Corollary}
\theoremstyle{definition}

\newtheorem{conj}{Conjecture}

\newtheorem{problem}{Problem}
\theoremstyle{remark}

\newcommand\nth{n^{\mathrm{th}}}

\newcommand\flr[1]{\left\lfloor #1\right\rfloor}
\newcommand\qint[1]{\left[ #1\right]_q}

\newcommand\Qbinom[3]{\genfrac{[}{]}{0pt}{}{#1}{#2}_{#3}}
\newcommand\qbinom[2]{\Qbinom{#1}{#2}{q}}
\newcommand\A{\mathcal{A}}

\newcommand\CM{\mathcal{CM}}
\newcommand\T{\mathcal{T}}
\newcommand\CT{\mathcal{CT}}

\newcommand\lm{\lambda/\mu}
\newcommand\RH{\widetilde{H}}

\newcommand\ta{\widetilde a}
\newcommand\tb{\widetilde b}
\newcommand\tc{\widetilde c}
\newcommand\td{\widetilde d}
\newcommand\tf{\widetilde f}
\newcommand\tv{\widetilde v}
\newcommand\abcd[1]{\alpha#1\beta#1\gamma#1\delta}
\newcommand\abcdpower{a^\alpha b^\beta c^\gamma d^\delta}
\newcommand\op{\overline P}
\newcommand\bg{{\beta+\gamma}}

\newcommand\RIGHT{\operatorname{right}}
\newcommand\BELOW{\operatorname{below}}
\newcommand\Cat{\operatorname{Cat}}
\newcommand\nara{\operatorname{Nara}}
\newcommand\inv{\operatorname{inv}}
\newcommand\wt{\operatorname{wt}}
\newcommand\blk{\operatorname{b}}

\newcommand\Mot{\operatorname{Mot}}
\newcommand\DSS{\operatorname{DSS}}

\newcommand\norm[1]{\lVert #1\rVert}

\newcommand\matching{\mathcal{M}}
\newcommand\cro{\operatorname{cr}}

\def\JV.{Josuat-Verg\`es}
\def\ywt.{$(y,q)$-weight}

\newcommand\ews{{}_8W_7}

\newcommand\twophione[3]{
  {}_2\phi_1 \left[ \left.
    \begin{array}{c}
      #1\\
      #2\\
    \end{array}
    \right| #3
    \right]
}

\psset{gridlabels=0pt,subgriddiv=1, gridwidth=.1pt,gridcolor=gray}
\psset{linewidth=.1pt, unit=.6cm}
\psset{unit=12pt}

\newcount\ax \newcount\ay
\newcount\bx \newcount\by
\newcount\cx \newcount\cy
\newcount\dx \newcount\dy
\def\cell(#1,#2)[#3]{
\ax=#2 \ay=#1
\multiply\ay by-1
\bx=\ax \by=\ay
\cx=\ax \cy=\ay
\dx=\ax \dy=\ay
\advance\bx by-1
\advance\dy by1
\advance\cx by-1
\advance\cy by1
\psline (\dx,\dy)(\ax,\ay)(\bx,\by)
\rput(\number\cx.5,
\ifnum\cy=0 -0.5\else\number\cy.5\fi){#3}
}
\def\gcell(#1,#2)[#3]{
\cell(#1,#2)[#3]
\psframe[linestyle=none,fillstyle=solid,fillcolor=gray!40!white](\ax,\ay)(\cx,\cy)
}
\def\scell(#1,#2)[#3]{
\cell(#1,#2)[#3]
\rput(\bx,\ay){\psline[linewidth=1.5pt](0,.5)(1,.5) \psline[linewidth=1.5pt](.5,0)(.5,1)}
}
\def\psrow(#1,#2){\multido{\i=1+1}{#2}{\cell(#1,\i)[]}}
\def\psdk#1{
\psline(#1,0)(0,0)(0,-#1)
\multido{\n=1+1}{#1}{
  \ax=#1 \advance\ax by 1
  \advance\ax by -\n
  \psrow(\n,\ax)}
}

\def\UPA{\rput(0,-.3){\psline[linewidth=1.5pt,arrowsize=.5, arrowlength=.6]{->}(0, .6)}}
\def\LTA{\rput(.3,0){\psline[linewidth=1.5pt,arrowsize=.5, arrowlength=.6]{->}(-.6, 0)}}

\def\HS(#1,#2){\rput(#1,#2){\psline[linewidth=1pt](0,0)(1,0)}}
\def\HSC(#1,#2){\rput(#1,#2){   
\pscurve [linewidth=1pt] (0,0)(.2,.1)(.4,-.1)(.6,.1)(.8,-.1)(1,0)
}}
\def\UP(#1,#2){\rput(#1,#2){\psline[linewidth=1pt](0,0)(1,1) 
}}
\def\DW(#1,#2){\rput(#1,#2){\psline[linewidth=1pt](0,0)(1,-1) 
}}
\def\MUP(#1,#2){\rput(#1,#2){\psline[linecolor=red, linewidth=2pt](0,0)(1,1)
}}
\def\MDW(#1,#2){\rput(#1,#2){\psline[linecolor=red, linewidth=2pt](0,0)(1,-1)
}}
\def\MHS(#1,#2){\rput(#1,#2){\psline[linecolor=red, linewidth=2pt](0,0)(1,0)}}
\def\MHSC(#1,#2){\rput(#1,#2){   
\pscurve [linecolor=red, linewidth=2pt] (0,0)(.2,.1)(.4,-.1)(.6,.1)(.8,-.1)(1,0)
}}

\def\wcirc(#1,#2){\rput(#2,-#1){\pscircle[fillstyle=solid](-.5,.5){.2}}}
\def\bcirc(#1,#2){\rput(#2,-#1){\pscircle*(-.5,.5){.2}}}
\def\dcirc(#1,#2){\rput(#2,-#1){\pscircle[linecolor=red](-.5,.5){.4}}}

\begin{document}

\begin{abstract}
 New formulas for the $n^{\mathrm{th}}$ moment $\mu_n(a,b,c,d;q)$ of the Askey-Wilson
polynomials are given. These are derived using analytic techniques, and by considering 
three combinatorial models for the moments: Motzkin paths, matchings, and 
staircase tableaux. A related positivity theorem is given and another one is conjectured.
\end{abstract}

\maketitle

\section{Introduction}
\label{sec:introduction}

The monic Askey-Wilson polynomials $P_n=P_n(x;a,b,c,d;q)$ are polynomials in 
$x$ of degree $n$ which depend upon five parameters $a$, $b$, $c$, $d$, and $q$. 
They may be defined by the
three-term recurrence $P_{n+1} = (x-b_n) P_n - \lambda_{n} P_{n-1}$ with
$P_{-1}=0$ and $P_0=1$ for $b_n =\frac{1}{2}(a+a^{-1}-(A_n+C_n))$ and
$\lambda_n=\frac{1}{4}A_{n-1}C_n$, where
\begin{align*}
A_n &= \frac{(1-abq^n)(1-acq^n)(1-adq^n)(1-abcdq^{n-1})}
{a(1-abcdq^{2n-1})(1-abcdq^{2n})},\\
C_n &= \frac{a(1-q^n)(1-bcq^{n-1})(1-bdq^{n-1})(1-cdq^{n-1})}
{(1-abcdq^{2n-2})(1-abcdq^{2n-1})}.
\end{align*}
We refer to \cite{GR} for the standard basic hypergeometric notation and 
for information about the Askey-Wilson polynomials.

In view of the above three-term recurrence relation, the Askey-Wilson
polynomials are orthogonal polynomials. An explicit absolutely continuous
measure \cite[Theorem 2.2]{AskeyWilson}, may be given for these polynomials. (We
assume here that $\max\{|a|,|b|,|c|,|d|\}<1.$) It is supported on $x\in [-1,1]$,
and with $x=\cos\theta,$ the measure is
\begin{equation}
  \label{eq:w}
  w(\cos\theta,a,b,c,d;q) = 
\frac{(e^{2i\theta})_\infty (e^{-2i\theta})_\infty}
{(ae^{i\theta})_\infty (ae^{-i\theta})_\infty (be^{i\theta})_\infty(be^{-i\theta})_\infty
(ce^{i\theta})_\infty (ce^{-i\theta})_\infty (de^{i\theta})_\infty (de^{-i\theta})_\infty}.
\nonumber
\end{equation}

The measure has total mass given by the {\it{Askey-Wilson integral}},
\begin{equation}
  \label{eq:AW_integral}
I_0(a,b,c,d) = \frac{(q)_\infty}{2\pi} \int_{0}^\pi w(\cos\theta,a,b,c,d;q) d\theta
=\frac{(abcd)_\infty}{(ab)_\infty (ac)_\infty (ad)_\infty (bc)_\infty (bd)_\infty (cd)_\infty}.
\end{equation}

The purpose of this paper is to study the $\nth$ moment $\mu_n(a,b,c,d;q)$ of the 
measure $w(x;a,b,c,d;q)$ for the Askey-Wilson polynomials
\begin{equation}
\mu_n(a,b,c,d;q) = C\int_{-1}^1 x^n w(x;a,b,c,d;q) \frac{dx}{\sqrt{1-x^2}}.
\nonumber
\end{equation}
for some normalization constant $C$. With the normalization of $\mu_0(a,b,c,d;q)=1$, 
(the explicit $C$ may be found from \eqref{eq:AW_integral}), 
the $\nth$ moment is known to be a rational function of $a$, $b$, $c$, $d$ and $q$. 
We shall give new explicit expressions  for $\mu_n(a,b,c,d;q)$ 
and study three combinatorial models for $\mu_n(a,b,c,d;q).$ 
One unusual feature of these results is the mixture of binomial and 
$q$-binomial terms in the explicit formulas. We shall see why this 
occurs, both analytically and combinatorially.

The simplest expression for $\mu_n(a,b,c,d;q)$ is a double sum 
(see \cite[Theorem~1.12]{CSSW})
\begin{equation}\label{AWWmomformula}
\mu_n(a,b,c,d;q) = \frac{1}{2^n} \sum_{m=0}^n \frac{(ab,ac,ad;q)_m}{(abcd;q)_m} q^m
\sum_{j=0}^m 
\frac{q^{-j^2}a^{-2j} (aq^j+q^{-j}/a)^n}{(q,q^{1-2j}/a^2;q)_j (q,q^{2j+1}a^2;q)_{m-j}}.
\end{equation}
However this expression is not obviously symmetric in $a$, $b$, $c$, and $d$, 
even though the Askey-Wilson polynomials $P_n(x;a,b,c,d;q)$ and the moments $\mu_n(a,b,c,d;q)$ are 
symmetric. Nor does it exhibit the correct poles of  $\mu_n(a,b,c,d;q)$ as a rational function. 
For $d=0$ the moments are polynomials in $a$, $b$, $c$, and $q$. 
We give new expressions for the 
moments $\mu_n(a,b,c,d;q)$, which are symmetric and polynomial when $d=0$ ,
see Theorem~\ref{thm:AWWd=0}. We also give a symmetric version for all $a,b,c,d$
in Theorem~\ref{thm:8W7}, although the polynomial dependence in $q$ is not clear.
We give new expressions for the moments $\mu_n(a,b,c,d;q)$ in the 
special case $b=-a$, $d=-c$, see Theorem~\ref{thm:AWWc=-a,d=-c} and 
Theorem~\ref{thm:AWWsymm}. We prove a new positivity 
theorem in Corollary~\ref{cor:flipcor}, and conjecture another one in 
Conjecture~\ref{conj:symmcon}.

Our second goal is to combinatorially study the moments $\mu_n(a,b,c,d;q)$ as
functions of $a$, $b$, $c$, and $d$.  We use three combinatorial models for this purpose. 
The moments for any set of orthogonal polynomials may be given as
weighted Motzkin paths \cite{V}. In this case $\mu_n(a,b,c,d;q)$ is
the generating function for Motzkin paths with weights which are rational
functions of $a,b,c,d$ and $q.$ We use this setup and a generalization of an
idea of D. Kim \cite{Kim1997} to combinatorially prove Corollary
~\ref{cor:AWWc=d=0} and Corollary~\ref{cor:bb=q/a} in
Section~\ref{sec:combinatorial-proof-}. A special combinatorial model for 
the Askey-Wilson integral, which evaluated the normalization constant $C,$ 
was given in \cite{ISV}. We modify this model appropriately to give a combinatorial 
model for some non-normalized moments in Theorem~\ref{thm:match}. 
We also give new explicit rational expressions for $\mu_n(a,b,c,d;q)$ so that 
$(abcd)_n\mu_n(a,b,c,d;q)$ are clearly polynomials in $a$, $b$, $c$, $d$ and $q$, 
see Theorem~\ref{thm:main} and Theorem~\ref{thm:main2}.

A third combinatorial model is given by Corteel and Williams \cite{Corteel2011}, who
give a combinatorial interpretation for the polynomial 
$2^n(abcd)_n\mu_n(a,b,c,d;q)$ using a rational transformation over the 
complex numbers of the parameters $a$, $b$,
$c$, and $d$ to parameters $\alpha$, $\beta$, $\gamma$, and $\delta$.  
With this transformation, $2^n(abcd)_n\mu_n(a,b,c,d;q)$ is a polynomial 
in $\alpha$, $\beta$, $\gamma$, and $\delta$ 
with positive integer coefficients, which has a combinatorial meaning. 
(See Section~\ref{sec:staircase-tableaux}). 
Using their ideas, explicit coefficients of certain terms in 
$2^n(abcd)_n\mu_n(a,b,c,d;q)$ as Catalan numbers are given in
Theorems~\ref{thm:highest} and \ref{thm:23}.  

We conclude in Section~\ref{sec:conn-with-other} with a summary of the special
cases of the Askey-Wilson polynomials whose moments had been considered
combinatorially. 

We give the first combinatorial proof of the formula of
Corteel et al. for the moments of $q$-Laguerre polynomials that is essentially
equivalent Corollary~\ref{cor:bb=q/a}, see Section~\ref{sec:conn-with-other}.
\JV. \cite{JV_PASEP} gave a different combinatorial proof of
Corollary~\ref{cor:AWWc=d=0}, but our proof is the first combinatorial proof of
Corollary~\ref{cor:bb=q/a}.

\section{Askey-Wilson moments}
\label{sec:askey-wilson-moments}

In this section we consider the moments $\mu_n(a,b,c,d;q)$ as functions of the 
parameters $a$, $b$, $c$, $d$ and $q$.  Our goal is to give new explicit formulas 
for these moments, using simple series and integral evaluations. 
We shall prove these specific results for the moments: 

\begin{itemize}
\item Theorem~\ref{thm:AWWmom} for the general case,
\item Theorem~\ref{thm:AWWd=0} and its Corollary~\ref{cor:AWWc=d=0} 
for a symmetric polynomial version when $d=0$, 
\item Theorem~\ref{thm:8W7}, a general symmetric polynomial version,
\item Corollary~\ref{cor:bb=q/a} and Corollary~\ref{cor:bb=-a}, as special cases of 
Corollary~\ref{cor:AWWc=d=0},
\item Theorem~\ref{thm:AWWsymm} and Theorem~\ref{thm:AWWc=-a,d=-c} for alternative formulas in the symmetric case, and
\item the positivity results Proposition~\ref{prop:partmatch} and 
Corollary~\ref{cor:bb=q/a}.
\end{itemize}

We use \eqref{AWWmomformula} and 
the method of proof of \eqref{AWWmomformula} which appears in \cite{CSSW} to prove 
these results. Positivity Conjecture~\ref{conj:symmcon} is also given here.

First we note the polynomial behavior of the moments. An explicit combinatorial 
version of this proposition, although with complex weights, is given 
in Proposition~\ref{prop:imagmom}.

\begin{prop}
\label{prop:polymom}
$2^n (abcd;q)_n \mu_n(a,b,c,d;q)$ is a polynomial in $a,b,c,d,q$ with integer coefficients.   
\end{prop}
\begin{proof}
  By \eqref{AWWmomformula}, $2^n (abcd;q)_n \mu_n(a,b,c,d;q)$ is a polynomial in
  $b$, $c$, and $d$ with possibly rational function coefficients in $a$ and $q$.
  By symmetry in $a$, $b$, $c$, and $d$, it is also a polynomial in $a$, thus a
  polynomial in $a$, $b$, $c$ and $d$, with coefficients which are rational
  function of $q$.  The poles of these rational functions must lie on either
  $|q|=1$ or $q=0$. On the other hand $\mu_n(a,b,c,d;q)$ is a polynomial in the
  three-term recurrence relation coefficients $b_n$ and $\lambda_n$ of
  Section~\ref{sec:introduction}, with integral coefficients. Since $b_n$ and
  $\lambda_n$ have no poles at these locations, $2^n (abcd;q)_n
  \mu_n(a,b,c,d;q)$ is also a polynomial in $q$, with integral coefficients.
\end{proof}

Using \eqref{AWWmomformula} we prove the first result, which 
exhibits  $2^n (abcd;q)_n\mu_n(a,b,c,d;q)$ as a symmetric polynomial in 
$b$, $c$, and $d.$

\begin{thm} \label{thm:AWWmom}
The Askey-Wilson moments are
\begin{multline*}
2^n\mu_n(a,b,c,d;q) =\sum_{m=0}^n 
\frac{(ab,ac,ad;q)_m}{(abcd;q)_m} (-q)^m
\sum_{s=0}^{n+1}  \left(\binom{n}{s} - \binom{n}{s-1}\right)\\
\times \sum_{p=0}^{n-2s-m} a^{-n+2s+2p}
\qbinom{m+p}{m} \qbinom{n-2s-p}{m} q^{(-n+2s+p)m+\binom{m}{2}}.
\end{multline*}
\end{thm}

\begin{proof}[Proof of Theorem~\ref{thm:AWWmom}]
We shall show that \eqref{AWWmomformula} implies Theorem~\ref{thm:AWWmom}.

Apply the binomial and $q$-binomial theorems to find that
the $j$-sum of \eqref{AWWmomformula} equals
\begin{align*}
& \sum_{j=0}^m q^{-j^2}a^{-2j} \cdot (-1)^j q^{2j^2-j-\binom{j}{2}} \qbinom{m}{j} 
\frac{q^{-jn}a^{-n} (1+a^2q^{2j})^n(1-a^2q^{2j})}{(a^2q^j;q)_{m+1} (q;q)_m}\\
=& \sum_{j=0}^m q^{-jn}a^{-n} (-1)^j q^{\binom{j}{2}} \qbinom{m}{j} 
\sum_{s=0}^{n+1} \left(\binom{n}{s}-\binom{n}{s-1}\right) a^{2s} q^{js}
\sum_{p=0}^\infty  \qbinom{m+p}{m} \frac{a^{2p} q^{jp}}{(q;q)_m}.
\end{align*}
Thus the $j$-sum is summable by the $q$-binomial theorem to obtain
\begin{multline*}
2^n\mu_n(a,b,c,d;q) =\sum_{m=0}^n 
\frac{(ab,ac,ad;q)_m}{(q,abcd;q)_m} q^m
\sum_{s=0}^{n+1}  \left(\binom{n}{s} - \binom{n}{s-1}\right)\\
\times \sum_{p=0}^{\infty} a^{-n+2s+2p}
\qbinom{m+p}{m} (q^{-n+2s+p};q)_m.
\end{multline*}
We extend the $p$ and $s$ sums from $-\infty$ to $\infty$ by declaring these
extended binomial and $q$-binomial coefficients to be zero.  Replacing $s$ by
$n+1-s$, and then $p$ by $p-n-1+2s$ for the second term of the difference gives
\begin{multline*}
2^n\mu_n(a,b,c,d;q)
=\sum_{m=0}^n 
\frac{(ab,ac,ad;q)_m}{(q,abcd;q)_m} q^m
\sum_{s=0}^{n+1} \binom{n}{s}\\
\times \sum_{p=-\infty}^{\infty} a^{-n+2s+2p}
\left( \qbinom{m+p}{m} (q^{-n+2s+p};q)_m
- \qbinom{m+p-n-1+2s}{m} (q^{p+1};q)_m \right). 
\end{multline*}
The first term is zero unless $p\ge0$ and ($-n+2s+p\le -m$ or $-n+2s+p\ge 1$).
The second term is zero unless $p-n-1+2s\ge0$ and ($p+1\le -m$ or $p+1\ge 1$).
If $p\ge0$ and $-n+2s+p\ge 1$, then the first and the second terms are nonzero
and equal.  Thus the difference is zero unless $0\le p \le n-2s-m$ for the first
term and $0\le p-n-1+2s$ and $p+1\le -m$ for the second term.  Undoing the
change of summation variables for the second term then gives 
Theorem~\ref{thm:AWWmom}.
\end{proof}

When $d=0$, we can obtain an explicit formula for the moments, which is a polynomial in 
each parameter $a$, $b$, $c$, and $q$.

\begin{thm} \label{thm:AWWd=0}
The Askey-Wilson moments for $d=0$ are
\begin{multline*}
2^n\mu_n(a,b,c,0;q) =
\sum_{k=0}^{n} \left(\binom{n}{\frac{n-k}2}-\binom{n}{\frac{n-k}2-1}\right)\\
\times 
\sum_{u+v+w+2t=k} a^u b^v c^w
(-1)^{t} q^{\binom{t+1}2} 
\qbinom{u+v+t}{v}
\qbinom{v+w+t}{w}
\qbinom{w+u+t}{u},
\end{multline*}
where the second sum is over all integers $0\le u,v,w\le k$ and $-k\le t\le k/2$
satisfying $u+v+w+2t=k$. 
\end{thm}

\begin{proof}[Proof of Theorem~\ref{thm:AWWd=0}]
We show that Theorem~\ref{thm:AWWd=0} follows from Theorem~\ref{thm:AWWmom}.

 In Theorem \ref{thm:AWWmom} expand the terms $(ab;q)_m$ and
$(ac;q)_m$ as polynomial in $a$, $b$ and $c$ by the $q$-binomial theorem to obtain
\[
2^n\mu_n(a,b,c,0;q) =
\sum_{s=0}^{n+1}  \left(\binom{n}{s} - \binom{n}{s-1}\right)
\sum_{p=0}^{n-2s} a^{-n+2s+2p} \cdot X,
\]
where
\[
X = \sum_{u,v\ge0} (-1)^{u+v} q^{\binom u2+\binom v2}
\frac{a^{u+v} b^u c^v}{(q;q)_u (q;q)_v}
\sum_{m=0}^{n-2s-p} 
\frac{(q^{p+1}, q^{-n+2s+p};q)_m q^m}{(q;q)_{m-u} (q;q)_{m-v}}.
\]
The $m$-sum is summable by the $q$-Vandermonde identity \cite[II.6]{GR} and we get
\[
X = \sum_{u,v\ge0} (-1)^{n-2s-p-u-v} 
q^{\binom{n-2s-p-u-v-1}2} a^{u+v} b^u c^v 
\qbinom{n-2s-p}{v} \qbinom{p+u}{u} \qbinom{p+v}{n-2s-p-u}.
\]
By replacing $p$ and $s$ with new summation variables $t=n-2s-p-u-v$,
$w=-n+2s+2p+u+v$, and $k=n-2s$, we obtain Theorem~\ref{thm:AWWd=0}.
\end{proof}

The special case $c=0$ of Theorem \ref{thm:AWWd=0} is equivalent to a 
result of  \JV. \cite [Theorem~6.1.1]{JV_PASEP}. 

\begin{cor} \label{cor:AWWc=d=0}
The Askey-Wilson moments for $c=d=0$ are
\begin{equation}
  \label{eq:ASC1}
2^n\mu_n(a,b,0,0;q) = 
\sum_{k=0}^{n} \left( \binom{n}{\frac{n-k}2}-\binom{n}{\frac{n-k}2-1}\right)
\sum_{u+v+2t=k} a^u b^v (-1)^{t} q^{\binom{t+1}2} \qbinom{u+v+t}{u,v,t},
\nonumber
\end{equation}
where the second sum is over all nonnegative integers $u,v,t$ satisfying
$u+v+2t=k$.
\end{cor}

In the next section we give a combinatorial proof of
Corollary~\ref{cor:AWWc=d=0}.  Finding a combinatorial proof of
Theorem~\ref{thm:AWWd=0} is still open.

\begin{problem}
  Find a combinatorial proof of Theorem~\ref{thm:AWWd=0}.
\end{problem}

There are two interesting special cases of Corollary~\ref{cor:AWWc=d=0} which 
are the next two corollaries.

\begin{cor}\label{cor:bb=q/a}
We have
\[
2^n\mu_n(a,q/a,0,0;q) = 
\sum_{k=0}^{n} \left( \binom{n}{\frac{n-k}2}-\binom{n}{\frac{n-k}2-1}\right)
(q/a)^k \sum_{i=0}^k a^{2i} q^{i(k-i-1)}.
\]
\end{cor}

\begin{cor}\label{cor:bb=-a}
We have $\mu_{2n+1}(a,-a,0,0;q)=0$ and 
\[
4^n\mu_{2n}(a,-a,0,0;q) = 
\sum_{k=0}^{n} \left( \binom{2n}{n-k}-\binom{2n}{n-k-1}\right)
\sum_{i=0}^k (-1)^i q^{\binom{i+1}2} (q;q^2)_{k-i} a^{2k-2i} \qbinom{2k-i}{i}. 
\]
\end{cor}

\begin{proof}[Proof of Corollary~\ref{cor:bb=q/a}]
Put $b=q/a$ in Corollary~\ref{cor:AWWc=d=0}, and substitute $v=k-2t-u$, and then 
$t+u=i$. The resulting $t$-sum is evaluable by the $q$-Vandermonde 
sum \cite[II.6]{GR}.
\end{proof}

Corollary~\ref{cor:bb=q/a} explicitly proves that $2^n \mu_n(a,q/a,0,0;q)$ is a
Laurent polynomial in $a$, with coefficients that are polynomials in $q$, with
positive integer coefficients. Corollary~\ref{cor:flipcor} its analogue for
$\mu_n(a,q/a,c,q/c;q).$

 \begin{proof}[Proof of Corollary~\ref{cor:bb=-a}]
We can rewrite \eqref{eq:ASC1} as
\[
2^n\mu_n(a,-a,0,0;q) = 
\sum_{k=0}^{n} \left( \binom{n}{\frac{n-k}2}-\binom{n}{\frac{n-k}2-1}\right)
\sum_{i=0}^{\flr{k/2}} (-1)^i q^{\binom{i+1}2} \qbinom{k-i}{k-2i} a^{k-2i}
\sum_{j=0}^{k-2i} (-1)^j  \qbinom{k-2i}{j}. 
\]  
Then we are done by the Gaussian formula \cite[Theorem 10, p.71]{AE}
\begin{equation}
  \label{eq:Gaussian_formula}
\sum_{i=0}^n  (-1)^i \qbinom{n}{i} =
\left\{ 
  \begin{array}{ll}
    0, & \mbox{if $n$ is odd,}\\
    (q;q^2)_{n/2}, & \mbox{if $n$ is even.}
  \end{array}
\right.
\nonumber
\end{equation}
\end{proof}

Another special case of the Askey-Wilson polynomials has a different expression for the moments. Consider $b=-a$ and $d=-c$, so that the Askey-Wilson measure is symmetric 
about the $y$-axis. In this case $b_n=0$, so the odd moments are zero, 
and the $2\nth$ moment has a shorter alternative expression.

\begin{thm} \label{thm:AWWc=-a,d=-c}
The non-zero Askey-Wilson moments for $b=-a$ and $d=-c$ are
\[
4^n\mu_{2n}(a,-a,c,-c;q)= \sum_{m=0}^n
\frac{(-a^2;q)_{2 m}(a^2c^2;q^2)_{m}}{(qa^2c^2;q^2)_m} q^{2m}
\sum_{j=0}^m  \frac{a^{-4j} q^{-2j^2} (aq^{j}+ a^{-1}q^{-j})^{2n}}
{(q^2,a^4q^{2+4j}; q^2)_{m-j} (q^2,a^{-4}q^{2-4j}; q^2)_{j}}.
\]
\end{thm}

\begin{proof}
 We use the idea of the proof of \cite[Proposition 3.1]{CSSW}. 
Let $y=2x^2-1=\cos(2\theta)$. 
We need to find the coefficients $A_m$ in the expansion
$$
\biggl(\frac{y+1}{2}\biggr)^n=x^{2n}=\sum_{m=0}^n A_m 
(a^2e^{2i\theta}, a^2e^{-2i\theta};q^2)_m.
$$ 
This may be done using the $q$-Taylor theorem, see 
\cite[Theorem 1.1] {IS2003}. The resulting Askey-Wilson 
integral is evaluated as in \cite[Lemma 3.2] {CSSW}. We do not give the details.
\end{proof}

Theorem~\ref{thm:AWWc=-a,d=-c} can be simplified in exactly the same way that 
\eqref{AWWmomformula} implied Theorem~\ref{thm:AWWmom} to obtain 
Theorem~\ref{thm:AWWsymm}. Again the details are not given.

\begin{thm} \label{thm:AWWsymm}
The non-zero Askey-Wilson moments for $b=-a$ and $d=-c$ are
\begin{multline*}
4^n \mu_{2n}(a,-a,c,-c;q) =
\sum_{m=0}^n \frac{(-a^2;q)_{2m}(a^2c^2;q^2)_m}
{(qa^2c^2;q^2)_m} (-q^2)^{m}\\
\times \sum_{s=0}^{2n+2} 
\left(\binom{2n+1}{s} - \binom{2n+1}{s-1} \right)
\sum_{p=0}^{n-m-s}  a^{-2n+4p+2s}
\Qbinom{m+p}{m}{q^2} 
\Qbinom{n-p-s}{m}{q^2} q^{-2m(n-p-s)+m(m-1)}. 
 \end{multline*}
\end{thm}

There is a corresponding positivity conjecture for the moments $\mu_{2n}(a,-a,c,-c;q)$. 
These moments are not polynomials, but Proposition~\ref{prop:polymom} implies that 
$4^n (a^2c^2;q)_{2n} \mu_{2n}(a,-a,c,-c;q)$ is a polynomial in $a$, $c$, and $q$. 
Half of the apparent poles of  $\mu_{2n}(a,-a,c,-c;q)$ do not occur.

\begin{prop} 
\label{prop:AWWsymmmom}The Askey-Wilson moments
\begin{eqnarray*}
\tau_{2n}(a^2,c^2)=4^n (qa^2c^2;q^2)_n \mu_{2n}(a,-a,c,-c;q)/(1-q)^n
\end{eqnarray*}
are polynomials in $a^2$, $c^2$ and $q$ with integer coefficients. Moreover 
the sum of the coefficients in $\tau_{2n}(a^2,c^2)$ is $2^{2n}(2n-1)(2n-3)\cdots 1$.
\end{prop}
\begin{proof}
After rescaling the Askey-Wilson polynomials, we have
$\tau_{2n}(a^2,c^2) = (qa^2c^2;q^2)_n \nu_{2n}$, 
where $\nu_{2n}$ are the $2\nth$ moments for the orthogonal polynomials defined by
\[
p_{n+1}(x) = x p_n(x) - \lambda_n p_{n-1}(x),
\]
where
\[
\lambda_n =
\frac{(1+a^2 q^{n-1}) (1+c^2 q^{n-1}) (1-a^2 c^2 q^{n-2})} 
{(1-a^2 c^2 q^{2n-3}) (1-a^2 c^2 q^{2 n-1})}[n]_q.
\]

From Viennot's theory \cite{V}, $\nu_{2n}$ is the generating function for 
Dyck paths from $(0,0)$ to $(2n,0)$ with weights given by a 
product of $\lambda_i's.$ So
\begin{eqnarray*}
\nu_{2n}= \frac{Q(a^2,c^2,q)}{\prod_{i=1}^n (1-a^2c^2q^{2i-1})^{n+1-i}}
\end{eqnarray*} 
for some polynomial $Q$ with integer coefficients. 
By Theorem \ref{thm:AWWc=-a,d=-c}, $\nu_{2n}$ has simple 
poles, as a function of $c^2$, possibly only at $(qa^2c^2;q^2)_n=0$.  So 
$(qa^2c^2;q^2)_n \nu_{2n}$ is a polynomial in $a^2$, $c^2$ and $q$, with integer coefficients.

For the final assertion, if $q=1$, $\lambda_n=n (1+a^2)(1+c^2)/(1-a^2c^2)$, which is the 
three term recurrence coefficient for a rescaled Hermite polynomial. The $2\nth$ moment for the Hermite 
polynomials is $(2n-1)(2n-3)\cdots 1$. 
\end{proof} 

\begin{conj} 
\label{conj:symmcon}
The coefficients of $\tau_{2n}(a^2,c^2)$  are non-negative integers.
\end{conj}


If $q=0$, one may show that $\tau_{2n}(a^2,c^2)$ is a non-negative polynomial by a combinatorial method. The sum of the coefficients is $2^{2n}$. It is a 
generating function for certain non-crossing complete matchings. 

Although simple, \eqref{AWWmomformula} does not clearly demonstrate the symmetry 
or polynomiality of $\mu_n(a,b,c,d;q)$ in all four parameters 
$a$, $b$, $c$ and $d$. We next give 
such a formula, which generalizes Theorem~\ref{thm:AWWmom}.

Let $A$ be an arbitrary parameter. Let
\[
\ews(m) = \ews(A^2/q; A/a,A/b,A/c,A/d,q^{-m}; q; abcdq^m).
\]
Note that $(aA,bA,cA,dA;q)_m \ \ews(m)$ is a symmetric polynomial in $a,b,c,d$:
\begin{equation}
\label{eq:bigEq}
\begin{aligned}
(aA,bA,cA,dA;q)_m \ & \ews(m) = \sum_{j=0}^m 
\frac{(A^2/q;q)_j}{(A^2q^m;q)_j} 
\frac{1-A^2q^{2j-1}} {1-A^2/q}\qbinom{m}{j}(-1)^j q^{\binom{j}{2}}\\
&\times (abcd)^j(A/a,A/b,A/c,A/d;q)_j (Aaq^j,Abq^j,Acq^j,Adq^j;q)_{m-j}. 
\end{aligned}
\end{equation}

Using Watson's transformation \cite[(III.17)]{GR} of an $\ews$ to a$\ _4\phi_3$, 
the following apparent rational function of $A$ and $q$ is in fact a 
polynomial in each of the parameters: $a$, $b$, $c$, $d$, $A$, and $q.$
\begin{equation}
\label{eq:polybigEq}
\begin{aligned}
\frac{(aA,bA,cA,dA;q)_m}{(A^2;q)_m} \ & \ews(m) = \sum_{j=0}^m 
\qbinom{m}{j}(cd)^j (A/c,A/d;q)_j (ab;q)_j \\
&\times (Aaq^j,Abq^j,cd;q)_{m-j}. 
\end{aligned}
\nonumber
\end{equation}

The next result gives a symmetric polynomial version 
for $2^n (abcd;q)_n \mu_n(a,b,c,d;q).$ Theorem~\ref{thm:8W7} is 
independent of $A$.

\begin{thm}
\label{thm:8W7}
\begin{multline*}
2^n\mu_n(a,b,c,d;q) =\sum_{m=0}^n 
\frac{(aA,bA,cA,dA;q)_m}{(A^2,abcd;q)_m} (-q)^m \ews(m)
\sum_{s=0}^{n+1}  \left(\binom{n}{s} - \binom{n}{s-1}\right)\\
\times \sum_{p=0}^{n-2s-m} A^{-n+2s+2p}
\qbinom{m+p}{m} 
\qbinom{n-2s-p}{m} q^{m(-n+2s+p)+\binom{m}{2}}.
\nonumber
 \end{multline*}
\end{thm}

\begin{proof} We again use the $q$-Taylor theorem and follow the proof of 
\cite[Proposition 3.1]{CSSW}.
We expand $x^n$ in terms of the basis $(Ae^{i\theta},Ae^{-i\theta};q)_k$. The 
resulting integral is a special case of the Nasrallah-Rahman integral 
\cite[(6.3.9), p. 158]{GR}. An analogue of \eqref{AWWmomformula} is obtained. 
As before, the same steps with the binomial and $q$-binomial theorems yield the 
stated result. 
\end{proof}

If $A=a$, then Theorem~\ref{thm:8W7} becomes Theorem~\ref{thm:AWWmom}.
 Theorem~\ref{thm:8W7} has one defect: not all of the powers of $q$ are positive
 due to the $q^{m(-n+2s+p)}$ term. The individual terms are Laurent polynomials in $q.$

If $A^2=q$, the $p$-sum in Theorem~\ref{thm:8W7} is evaluable by the 
$q$-Vandermonde sum \cite[II.6]{GR}.

\begin{cor}
\label{weird8W7} If $A^2=q$,
\begin{multline*}
2^n\mu_n(a,b,c,d;q) =\sum_{m=0}^n 
\frac{(aA,bA,cA,dA;q)_m}{(A^2,abcd;q)_m} (-q)^m \ews(m)
\sum_{s=0}^{n+1}  \left(\binom{n}{s} - \binom{n}{s-1}\right)\\
\times A^{-n+2s}
\qbinom{n+m+1-2s}{2m+1} 
q^{-nm + 2sm +\binom{m}{2}}.
 \end{multline*}
\end{cor}

An explicit polynomial version of Theorem \ref{thm:AWWd=0}, which is unwieldy,  
may be given by expanding each of the eight factors in \eqref{eq:bigEq} by the 
$q$-binomial theorem. The resulting formula would appear to be too complex to be useful.

\section{Weighted Motzkin paths}
\label{sec:combinatorial-proof-}

In this section we give combinatorial proofs of Corollary~\ref{cor:AWWc=d=0} and
Corollary~\ref{cor:bb=q/a}, the explicit polynomial versions of the moments 
with 2 parameters. The combinatorial model we shall use is Motzkin paths, based 
on the three-term recurrence formula, 

A sketch of the proof is as follows.  We first interpret the moment $2^n
\mu_n(a,b,0,0;q)$ as a weighted sum of Motzkin paths. Then using Penaud's
decomposition \cite{Pen} we can decompose a weighted Motzkin path into a pair of
paths: a Dyck prefix and another weighted Moztkin path. The Dyck prefixes give a 
difference of binomial coefficients. We map the new weighted
Moztkin path to a new object: doubly striped skew shapes. These objects are a
generalization of striped skew shapes introduced by D. Kim \cite{Kim1997} in
order to prove the moment formula for Al-Salam-Carlitz polynomials. We then find
a sign-reversing involution on doubly striped skew shapes and show that the
fixed points have a weighted sum equal to a $q$-trinomial coefficient. Thus 
we have found the differences of binomial coeffcients and the $q$-trinomial which appear in 
Corollary~\ref{cor:AWWc=d=0}. For Corollary~\ref{cor:bb=q/a}, we find a further cancellation on the doubly striped skew shapes which leaves only one fixed point for given size.

A \emph{Motzkin path} is a lattice path in $\mathbb{N}\times\mathbb{N}$ from
$(0,0)$ to $(n,0)$ consisting of up steps $(1,1)$, down steps $(1,-1)$, and
horizontal steps $(1,0)$. We say that the \emph{level} of a step is $i$ if it is
an up step or a down step between the lines $y=i-1$ and $y=i$, or it is a
horizontal step on the line $y=i$.  A \emph{weighted Motzkin path} is a Motzkin
path in which each step has a certain weight. The \emph{weight} $\wt(p)$ of a
weighted Motzkin path $p$ is the product of the weights of all steps.  Note that
the level of an up step or a down step is at least $1$ and the level of a
horizontal step may be $0$.
 
Let $\Mot_n(a,b)$ denote the set of weighted Motzkin paths of length $n$ such that
\begin{itemize}
\item the weight of an up step of level $i$ is either $q^i$ or $-1$,
\item the weight of a down step of level $i$ is either $abq^{i-1}$ or $-1$,
\item the weight of a horizontal step of level $i$ is either $aq^i$ or $bq^{i}$.
\end{itemize}
Then by Viennot's theory \cite{V} we have
\begin{equation}
  \label{eq:1}
2^n \mu_n(a,b,0,0;q)= \sum_{P\in \Mot_n(a,b)} \wt(P). 
\end{equation}

We define $\Mot_n^*(a,b)$ to be the set of weighted Motzkin paths in
$\Mot_n(a,b)$ such that there is no peak of weight $1$, that is, an up step of
weight $-1$ immediately followed by a down step of weight $-1$.

By the same idea that is a variation of Penaud's decomposition as in
\cite[Proposition~5.1]{JV2010}, we have
\begin{equation}
  \label{eq:2}
\sum_{P\in \Mot_n(a,b)} \wt(P) = \sum_{k=0}^{n} 
\left(\binom{n}{\frac{n-k}2}-\binom{n}{\frac{n-k}2-1}\right)
\sum_{P\in \Mot_k^*(a,b)} \wt(P).
\end{equation}

By \eqref{eq:1} and \eqref{eq:2}, Corollary~\ref{cor:AWWc=d=0}
and Corollary~\ref{cor:bb=q/a} can be restated as follows.

\begin{thm}\label{thm:mot}
We have
  \begin{align*}
  \sum_{P\in \Mot_k^*(a,b)} \wt(P) &= 
\sum_{u+v+2t=k} a^u b^v (-1)^{t} q^{\binom{t+1}2} \qbinom{u+v+t}{u,v,t},\\
\sum_{P\in \Mot_k^*(a,q/a)} \wt(P) &=  (q/a)^k \sum_{i=0}^k a^{2i}
q^{i(k-i-1)}. 
  \end{align*}
\end{thm}

We note that the second identity in Theorem~\ref{thm:mot} is equivalent to
Proposition~5 in \cite{CJVPR} which is used to prove the formula
\eqref{eq:corteel} for the moments of $q$-Laguerre polynomials.  Corteel et
al.'s proof of \cite{CJVPR} is combinatorial except for the proof of
Proposition~5 in \cite{CJVPR}.  In this section we prove the above theorem
combinatorially, thus providing the first combinatorial proof of
\eqref{eq:corteel}.

In order to give combinatorial proofs of the two formulas in
Theorem~\ref{thm:mot} we introduce doubly striped skew shapes. These objects are
a generalization of striped skew shapes introduced by D.~Kim \cite{Kim1997}.

A \emph{doubly striped skew shape} of size $m\times n$ is a quadruple
$(\lambda,\mu,W,B)$ of partitions $\mu\subset \lambda \subset (n^m)$ and a set
$W$ of white stripes and a set $B$ of black stripes with $W\cap
B=\emptyset$. Here, a \emph{white stripe} is a diagonal set $S$ of $\lm$ such
that $\lm$ contains neither the cell to the left of the leftmost cells of $S$
nor the cell below the rightmost cell of $S$, where a diagonal set means a set
of cells in row $r+i$ and column $s+i$ for $i=1,2,\dots,p$ for some integers
$r,s,p$. Similarly, a \emph{black stripe} is a diagonal set $S$ of $\lm$ such
that $\lm$ contains neither the cell above the leftmost cell of $S$ and the cell
to the right of the rightmost cell $S$. We will call a cell in a white stripe
(resp.~black stripe) a \emph{white dot} (resp.~\emph{black dot}).

 Let $\DSS(m,n)$ denote the
set of doubly striped skew shapes of size $m\times n$. We define the
\emph{weight} of $(\lambda,\mu,W,B)\in \DSS(m,n)$ to be
\begin{equation}
  \label{eq:wt}
\wt_{a,b}(\lambda,\mu,W,B) = a^m b^n (-1)^{|W|+|B|} 
q^{|\lm|-\norm{W}-\norm{B}} (q/ab)^{|W|},
\end{equation}
where $\norm{W}$ and $\norm{B}$ are the total numbers of white dots and black
dots respectively.

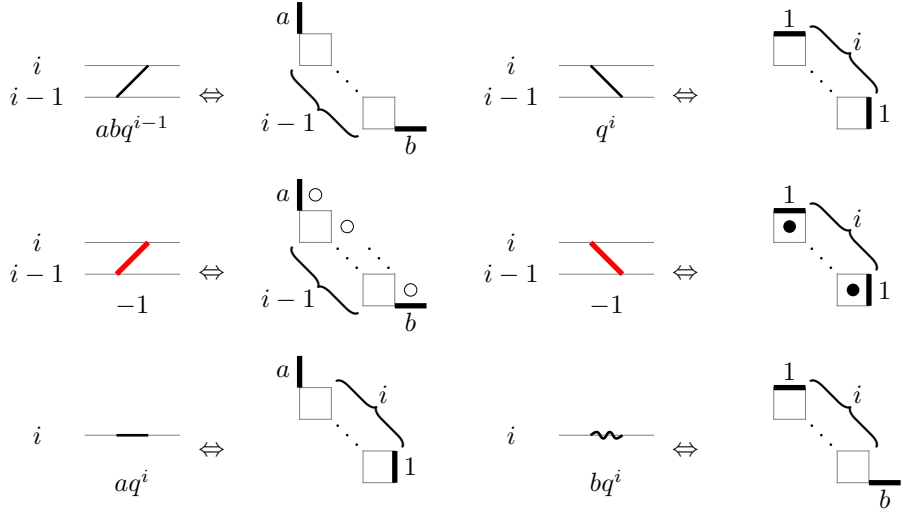
\begin{figure}
  \centering
\begin{pspicture}(-2.5,-2)(3,2)
\psline[linecolor=gray](0,1)(3,1)
\psline[linecolor=gray](0,0)(3,0)
\psset{linewidth=1pt}
\UP(1,0) 
\rput(-1.5,0){$i-1$}
\rput(-1.5,1){$i$}
\rput(1.5,-1){$abq^{i-1}$}
\end{pspicture}%
\begin{pspicture}(0,0)(2,4)
\rput(1,2){$\Leftrightarrow$}
\end{pspicture}
\begin{pspicture}(-.5,0)(5,5)
 \psgrid(1,3)(2,4) \psgrid(3,1)(4,2)
\psset{linewidth=2pt,dotsize=1pt}
\rput(-1.5,1.5){\psdot(4,1) \psdot(3.7,1.3) \psdot(4.3,.7)}
\psline(1,5)(1,4) \psline(5,1)(4,1)
\rput{-45}(1.7,1.7){\psscaleboxto(3,.5){\rotateleft{\{}}}
\rput[tr](1.5,1.5){$i-1$}
\rput[r](0.7,4.5){$a$}
\rput(4.5,.5){$b$}
\end{pspicture}\qquad
\begin{pspicture}(-2.5,-2)(3,2)
\psline[linecolor=gray](0,1)(3,1)
\psline[linecolor=gray](0,0)(3,0)
\psset{linewidth=1pt}
\DW(1,1) 
\rput(-1.5,0){$i-1$}
\rput(-1.5,1){$i$}
\rput(1.5,-1){$q^{i}$}
\end{pspicture}%
\begin{pspicture}(0,0)(2,4)
\rput(1,2){$\Leftrightarrow$}
\end{pspicture}
\begin{pspicture}(-.5,0)(5,5)
 \psgrid(1,3)(2,4) \psgrid(3,1)(4,2)
\psset{linewidth=2pt,dotsize=1pt}
\rput(-1.5,1.5){\psdot(4,1) \psdot(3.7,1.3) \psdot(4.3,.7)}
\psline(2,4)(1,4) \psline(4,2)(4,1)
\rput{-45}(3.3,3.3){\psscaleboxto(3,.5){\rotateright{\{}}}
\rput[bl](3.5,3.5){$i$}
\rput(1.5,4.5){$1$}
\rput(4.5,1.5){$1$}
\end{pspicture}

\begin{pspicture}(-2.5,-2)(3,2)
\psline[linecolor=gray](0,1)(3,1)
\psline[linecolor=gray](0,0)(3,0)
\psset{linewidth=1pt}
\MUP(1,0) 
\rput(-1.5,0){$i-1$}
\rput(-1.5,1){$i$}
\rput(1.5,-1){$-1$}
\end{pspicture}%
\begin{pspicture}(0,0)(2,4)
\rput(1,2){$\Leftrightarrow$}
\end{pspicture}
\begin{pspicture}(-.5,0)(5,5.5)
\rput(2,4){\wcirc(0,0)} \rput(3,3){\wcirc(0,0)} \rput(5,1){\wcirc(0,0)}
 \psgrid(1,3)(2,4) \psgrid(3,1)(4,2)
\psset{linewidth=2pt,dotsize=1pt}
\rput(-1.5,1.5){\psdot(4,1) \psdot(3.7,1.3) \psdot(4.3,.7)}
\rput(-.5,1.5){\psdot(4,1) \psdot(3.7,1.3) \psdot(4.3,.7)}
\psline(1,5)(1,4) \psline(5,1)(4,1)
\rput{-45}(1.7,1.7){\psscaleboxto(3,.5){\rotateleft{\{}}}
\rput[tr](1.5,1.5){$i-1$}
\rput[r](0.7,4.5){$a$}
\rput(4.5,.5){$b$}
\end{pspicture}\qquad
\begin{pspicture}(-2.5,-2)(3,2)
\psline[linecolor=gray](0,1)(3,1)
\psline[linecolor=gray](0,0)(3,0)
\psset{linewidth=1pt}
\MDW(1,1) 
\rput(-1.5,0){$i-1$}
\rput(-1.5,1){$i$}
\rput(1.5,-1){$-1$}
\end{pspicture}%
\begin{pspicture}(0,0)(2,4)
\rput(1,2){$\Leftrightarrow$}
\end{pspicture}
\begin{pspicture}(-.5,0)(5,5.5)
\rput(2,3){\bcirc(0,0)} \rput(4,1){\bcirc(0,0)}
 \psgrid(1,3)(2,4) \psgrid(3,1)(4,2)
\psset{linewidth=2pt,dotsize=1pt}
\rput(-1.5,1.5){\psdot(4,1) \psdot(3.7,1.3) \psdot(4.3,.7)}
\psline(2,4)(1,4) \psline(4,2)(4,1)
\rput{-45}(3.3,3.3){\psscaleboxto(3,.5){\rotateright{\{}}}
\rput[bl](3.5,3.5){$i$}
\rput(1.5,4.5){$1$}
\rput(4.5,1.5){$1$}
\end{pspicture}

\begin{pspicture}(-2.5,-2)(3,2)
\psline[linecolor=gray](0,.5)(3,.5)
\HS(1,.5) 
\rput(-1.5,.5){$i$}
\rput(1.5,-1){$aq^{i}$}
\end{pspicture}%
\begin{pspicture}(0,0)(2,4)
\rput(1,2){$\Leftrightarrow$}
\end{pspicture}
\begin{pspicture}(-.5,0)(5,5.)
 \psgrid(1,3)(2,4) \psgrid(3,1)(4,2)
\psset{linewidth=2pt,dotsize=1pt}
\rput(-1.5,1.5){\psdot(4,1) \psdot(3.7,1.3) \psdot(4.3,.7)}
\psline(1,4)(1,5) \psline(4,2)(4,1)
\rput{-45}(3.3,3.3){\psscaleboxto(3,.5){\rotateright{\{}}}
\rput[bl](3.5,3.5){$i$}
\rput[r](.7,4.5){$a$}
\rput(4.5,1.5){$1$}
\end{pspicture}\qquad
\begin{pspicture}(-2.5,-2)(3,2)
\psline[linecolor=gray](0,.5)(3,.5)
\HSC(1,.5)
\rput(-1.5,.5){$i$}
\rput(1.5,-1){$bq^{i}$}
\end{pspicture}%
\begin{pspicture}(0,0)(2,4)
\rput(1,2){$\Leftrightarrow$}
\end{pspicture}
\begin{pspicture}(-.5,0)(5,5.5)
 \psgrid(1,3)(2,4) \psgrid(3,1)(4,2)
\psset{linewidth=2pt,dotsize=1pt}
\rput(-1.5,1.5){\psdot(4,1) \psdot(3.7,1.3) \psdot(4.3,.7)}
\psline(2,4)(1,4) \psline(5,1)(4,1)
\rput{-45}(3.3,3.3){\psscaleboxto(3,.5){\rotateright{\{}}}
\rput[bl](3.5,3.5){$i$}
\rput(1.5,4.5){$1$}
\rput(4.5,.5){$b$}
\end{pspicture}
 \caption{Converting a weighted Motzkin path to a doubly striped skew shape.}
  \label{fig:Mot_to_Dyck}
\end{figure}

\begin{figure}
  \centering
\begin{pspicture}(0,-1)(11,3)
  \psgrid (0,0)(11,3)
  \MUP(0,0) \UP(1,1) \UP(2,2) \MDW(3,3) \HSC(4,2) \MUP(5,2) \HS(6,3) 
  \MDW(7,3) \MDW(8,2) \HSC(9,1) \DW(10,1)
\end{pspicture}
\begin{pspicture}(-4,0)(6,-5)
\rput(-2,-2.5){$\Leftrightarrow$}
 \gcell(1,1)[] \gcell(1,2)[] \gcell(2,1)[] \gcell(2,2)[]
 \gcell(2,6)[] \gcell(3,6)[] \gcell(4,6)[] \gcell(5,6)[]  \gcell(5,4)[] \gcell(5,5)[] 
  \bcirc(1,4) \bcirc(2,5)
  \bcirc(3,1) \bcirc(4,2) \bcirc(5,3) 
  \wcirc(5,1) \wcirc(2,3) \wcirc(3,4) \wcirc(4,5) 
\bcirc(1,3) \bcirc(2,4) \bcirc(3,5) 
  \psgrid (0,0)(6,-5)
\end{pspicture}
\caption{An example of the bijection $\rho$. The up steps of weight $y$ and the
  down steps of weight $1$ are colored red.}
  \label{fig:dsss}
\end{figure}
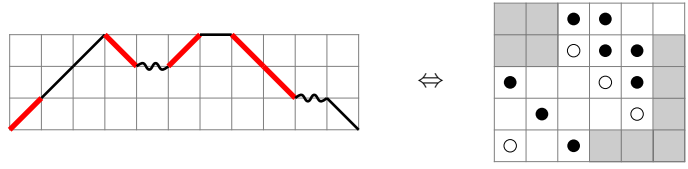

There is a simple correspondence between $\Mot_k^*(a,b)$ and doubly striped skew
shapes.

\begin{lem}\label{lem:XDSS}
We have
\[
\sum_{P\in \Mot_k^*(a,b)} \wt(P) = \sum_{i=0}^k \sum_{S\in \DSS(i,k-i)}
\wt_{a,b}(S).
\]
\end{lem}
\begin{proof}
  It is enough to show that there is a weight-preserving bijection $\rho:
  \Mot_k^*(a,b) \to \bigcup_{i=0}^k \DSS(i,k-i)$.  Let $P\in \Mot_k^*(a,b)$. We
  will construct an upper path and a lower path which determine two partitions
  $\lambda$ and $\mu$ respectively. The upper path and the lower path start at
  the origin. For each step of $P$, we add one step to the two lattice paths as
  follows. If the step is an up step of weight $abq^{i-1}$, we add a north step
  of weight $a$ to the upper path and an east step of weight $b$ to the lower
  path. If the step is an up step of weight $-1$, we add a north step of weight
  $a$ to the upper path and an east step of weight $b$ to the lower path, and we
  make a white stripe between these two steps, see
  Figure~\ref{fig:Mot_to_Dyck}. Similarly we add one step to the upper and lower
  paths for other type of step in $P$ as shown in
  Figure~\ref{fig:Mot_to_Dyck}. Then we define $\rho(P)$ to be the resulting
  diagram.  See Figure~\ref{fig:dsss} for an example of $\rho$. It is easy to
  see that $\rho$ is a weight-preserving bijection.
\end{proof}

By Lemma~\ref{lem:XDSS}, Theorem~\ref{thm:mot} is equivalent to the following
proposition.

\begin{prop}\label{prop:dss}
For $k\geq0$, we have
\begin{align}
  \label{eq:dss1}
\sum_{i=0}^k \sum_{S\in \DSS(i,k-i)} \wt_{a,b}(S) &=
\sum_{u+v+2t=k} a^u b^v (-1)^{t} q^{\binom{t+1}2} \qbinom{u+v+t}{u,v,t},\\
\label{eq:dss2}
\sum_{i=0}^k \sum_{S\in \DSS(i,k-i)} \wt_{a,q/a}(S) &=
(q/a)^k \sum_{i=0}^k a^{2i}q^{i(k-i-1)}. 
\end{align}
\end{prop}

In order to prove Proposition~\ref{prop:dss} we need D. Kim's sign-reversing
involution. We now recall his involution. By a sign-reversing involution on
$\DSS(m,n)$ we mean an involution $f:\DSS(m,n)\to \DSS(m,n)$ such that if $S\in
\DSS(m,n)$ is not a fixed point, i.e. $f(S)\ne S$, then $\wt(f(S)) = -\wt(S)$.

Let $\DSS(m,n;\alpha,\beta,\gamma,\delta)$ denote the set of $(\lambda,\mu,W,B)
\in \DSS(m,n)$ such that $\lambda=\alpha$, $\mu=\beta$, $W=\gamma$, and
$B=\delta$, where $\alpha$ can be ``$-$'' and in this case there is no
restriction on $\lambda$, and it is similar for $\beta$, $\gamma$, and
$\delta$. For example, $\DSS(m,n; \lambda, -, \emptyset, -)$ is the set of
$(\lambda,\mu, \emptyset,B) \in \DSS(m,n)$ for all possible $\mu$ and $B$.

The following is proved by D. Kim \cite[Theorem~3.2]{Kim1997}. Since we will use
his proof we include it here.

\begin{lem}\label{lem:dkim}
For a partition $\lambda\subset (n^m)$, we have
\[
\sum_{S\in \DSS(m,n;\lambda,-,\emptyset,-)} \wt_{a,b}(S) = \sum_{S\in
  \DSS(m,n;\lambda,\emptyset,\emptyset, \emptyset)} \wt_{a,b}(S) = a^m b^n
q^{|\lambda|}.
\]  
\end{lem}
\begin{proof}
  We construct a sign-reversing involution $\phi:\DSS(m,n;\lambda,-,
  \emptyset,-)\to \DSS(m,n;\lambda,-, \emptyset,-)$ as follows.

  Let $(\lambda,\mu,\emptyset,B)\in \DSS(m,n;\lambda,-,\emptyset,-)$.  Suppose
  $\mu$ has $r$ nonempty rows and the lowest row containing a black dot is row
  $s$. If there is no black stripe, then we let $s=0$. If $r=s=0$, equivalently
  $\mu=B=\emptyset$, we define $(\lambda,\mu,\emptyset,B)$ to be a fixed
  point. Otherwise, we define $\phi(\lambda,\mu,\emptyset,B)$ as follows.

  Case 1: If $s\geq r$, move the dots in the stripe containing a dot in row $s$
  (there must be a unique such stripe) all the way to the left and fill the
  cells containing these dots.

  Case 2: If $s<r$, we remove the leftmost cells of the last $t$ rows in $\mu$
  and build a black stripe of size $t$ ending in row $r$ for the unique integer
  $t$ for which this operation is possible.

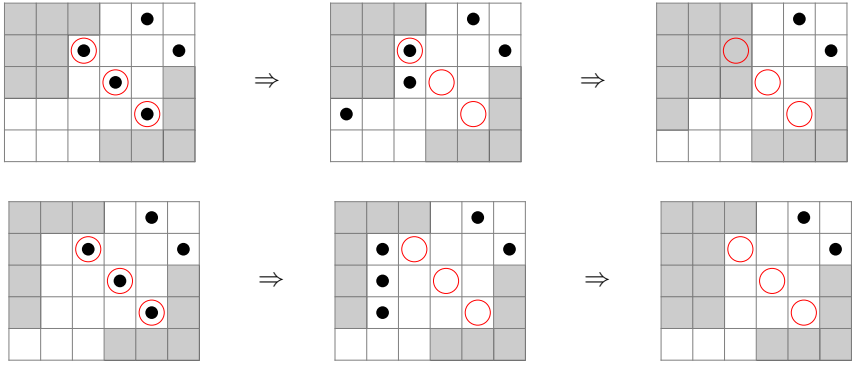
\begin{figure}
  \centering
\begin{pspicture}(0,0)(6,-5)
\gcell(1,1)[] \gcell(1,2)[] \gcell(2,1)[] \gcell(2,2)[] \gcell(1,3)[] \gcell(3,1)[] 
\gcell(3,6)[] \gcell(4,6)[] \gcell(5,6)[]  \gcell(5,4)[] \gcell(5,5)[] \gcell(3,2)[]
\bcirc(2,3) \bcirc(3,4) \bcirc(4,5) \bcirc(2,6) \bcirc(1,5)
\dcirc(2,3) \dcirc(3,4) \dcirc(4,5)
  \psgrid (0,0)(6,-5)
\end{pspicture}
\begin{pspicture}(-4,0)(6,-5)
\rput(-2,-2.5){$\Rightarrow$}
\gcell(1,1)[] \gcell(1,2)[] \gcell(2,1)[] \gcell(2,2)[] \gcell(1,3)[] \gcell(3,1)[] 
\gcell(3,6)[] \gcell(4,6)[] \gcell(5,6)[]  \gcell(5,4)[] \gcell(5,5)[] \gcell(3,2)[]
\bcirc(2,3) \bcirc(4,1) \bcirc(3,3) \bcirc(2,6) \bcirc(1,5)
\dcirc(2,3) \dcirc(3,4) \dcirc(4,5)
  \psgrid (0,0)(6,-5)
\end{pspicture}
\begin{pspicture}(-4,0)(6,-5)
\rput(-2,-2.5){$\Rightarrow$}
\gcell(1,1)[] \gcell(1,2)[] \gcell(2,1)[] \gcell(2,2)[] \gcell(1,3)[] \gcell(3,1)[] 
\gcell(3,6)[] \gcell(4,6)[] \gcell(5,6)[]  \gcell(5,4)[] \gcell(5,5)[] \gcell(3,3)[]
\gcell(2,3)[] \gcell(3,2)[] \gcell(4,1)[] 
\bcirc(2,6) \bcirc(1,5)
\dcirc(2,3) \dcirc(3,4) \dcirc(4,5)
  \psgrid (0,0)(6,-5)
\end{pspicture}
\vspace{.5cm}

\begin{pspicture}(0,0)(6,-5)
\gcell(1,1)[] \gcell(1,2)[] \gcell(2,1)[] \gcell(1,3)[] 
\gcell(3,1)[] \gcell(4,1)[] 
\gcell(3,6)[] \gcell(4,6)[] \gcell(5,6)[]  \gcell(5,4)[] \gcell(5,5)[] 
\bcirc(2,6) \bcirc(1,5)
\dcirc(2,3) \dcirc(3,4) \dcirc(4,5)
\bcirc(2,3) \bcirc(3,4) \bcirc(4,5)
 \psgrid (0,0)(6,-5)
\end{pspicture}
\begin{pspicture}(-4,0)(6,-5)
\rput(-2,-2.5){$\Rightarrow$}
\gcell(1,1)[] \gcell(1,2)[] \gcell(2,1)[] \gcell(1,3)[] 
\gcell(3,6)[] \gcell(4,6)[] \gcell(5,6)[]  \gcell(5,4)[] \gcell(5,5)[] 
\gcell(3,1)[] \gcell(4,1)[] 
\bcirc(2,6) \bcirc(1,5)
\dcirc(2,3) \dcirc(3,4) \dcirc(4,5)
\bcirc(2,2) \bcirc(3,2) \bcirc(4,2)
  \psgrid (0,0)(6,-5)
\end{pspicture}
\begin{pspicture}(-4,0)(6,-5)
\rput(-2,-2.5){$\Rightarrow$}
\gcell(1,1)[] \gcell(1,2)[] \gcell(2,1)[] \gcell(2,2)[] \gcell(1,3)[] \gcell(3,1)[] 
\gcell(3,6)[] \gcell(4,6)[] \gcell(5,6)[]  \gcell(5,4)[] \gcell(5,5)[] 
\gcell(4,1)[] \gcell(3,2)[] \gcell(4,2)[] 
\bcirc(2,6) \bcirc(1,5)
\dcirc(2,3) \dcirc(3,4) \dcirc(4,5)
  \psgrid (0,0)(6,-5)
\end{pspicture}
\caption{Two examples of the involution $\phi$.}
  \label{fig:dkim_inv}
\end{figure}
 
See Figure~\ref{fig:dkim_inv} for some examples of this map.  It is easy to
check that $\phi$ is a sign-reversing involution on
$\DSS(m,n;\lambda,-,\emptyset,-) $ with only one fixed point
$(\lambda,\emptyset,\emptyset,\emptyset)$, which implies the desired identity.
\end{proof}

The involution $\phi$ of D. Kim cannot be directly applied to doubly striped skew
shapes with white stripes. However, in the proof of the next lemma, we show that
there is a simple way to extend his involution when we have white stripes.

\begin{lem}\label{lem:general_dkim}
We have
\[
\sum_{S\in \DSS(m,n)} \wt_{a,b}(S) = 
\sum_{S\in \DSS(m,n;-,\emptyset,-,\emptyset)} \wt_{a,b}(S).
\]  
\end{lem}
\begin{proof}
  Let $(\lambda,\mu,W,B)\in \DSS(m,n)$. Ignoring the white stripes, we apply the
  involution $\phi$. Then the newly added cells may overlap with white
  stripes. In this case we slide each overlapped white stripe one step to the
  right. If the rightmost dot in the stripe cannot be moved, then we delete this
  dot and the cell containing it from $\lambda$ and slide the white strip. See
  Figure~\ref{fig:slide}. We define $\psi(\lambda,\mu,W,B)$ to be the resulting
  doubly striped skew shape.

  One can easily see that $\psi$ is a sign-reversing involution on $\DSS(m,n)$
  with fixed point set $\DSS(m,n;-,\emptyset,-,\emptyset)$. Thus we get the
  desired identity.
\end{proof}

\begin{figure}
  \centering
\begin{pspicture}(0,0)(5,-5)
  \gcell(1,1)[] \wcirc(1,1) \wcirc(2,2) \wcirc(3,3) \wcirc(4,4)
\gcell(5,4)[] \gcell(5,5)[]
\psgrid(0,0)(1,-1) \psgrid(3,-4)(5,-5)
\end{pspicture}%
\begin{pspicture}(-2,0)(5,-5)
\rput(-1,-2.5){$\Rightarrow$}
  \gcell(1,1)[] \wcirc(1,2) \wcirc(2,3) \wcirc(3,4) \wcirc(4,5)
\gcell(5,4)[] \gcell(5,5)[]
\psgrid(0,0)(1,-1) \psgrid(3,-4)(5,-5)
\end{pspicture} 
\begin{pspicture}(-5,0)(5,-5)
\rput(-2.5,-2.5){and}
  \gcell(1,1)[] \wcirc(1,1) \wcirc(2,2) \wcirc(3,3) \wcirc(4,4)
\gcell(5,4)[] \gcell(5,5)[] \gcell(4,5)[]
\psgrid(0,0)(1,-1) \psgrid(3,-4)(5,-5) \psgrid(4,-3)(5,-4)
\end{pspicture}%
\begin{pspicture}(-2,0)(5,-5)
\rput(-1,-2.5){$\Rightarrow$}
  \gcell(1,1)[] \wcirc(1,2) \wcirc(2,3) \wcirc(3,4) 
\gcell(5,4)[] \gcell(5,5)[] \gcell(4,5)[] \gcell(4,4)[]
\psgrid(0,0)(1,-1) \psgrid(3,-3)(5,-5) 
\end{pspicture}
  \caption{Slide rules}
  \label{fig:slide}
\end{figure}
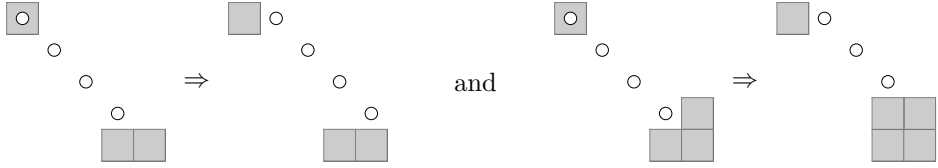

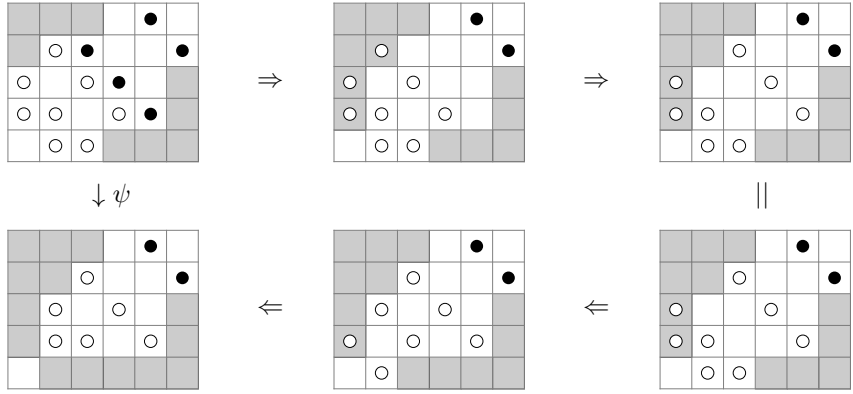
\begin{figure}
  \centering
\begin{pspicture}(0,0)(6,-5)
\gcell(1,1)[] \gcell(1,2)[] \gcell(2,1)[] \gcell(1,3)[] 
\gcell(3,6)[] \gcell(4,6)[] \gcell(5,6)[]  \gcell(5,4)[] \gcell(5,5)[] 
\bcirc(2,6) \bcirc(1,5)
\bcirc(2,3) \bcirc(3,4) \bcirc(4,5)
\wcirc(2,2) \wcirc(3,3) \wcirc(4,4)
\wcirc(3,1) \wcirc(4,2) \wcirc(5,3)
\wcirc(4,1) \wcirc(5,2)
 \psgrid (0,0)(6,-5)
\end{pspicture}
\begin{pspicture}(-4,0)(6,-5)
\rput(-2,-2.5){$\Rightarrow$}
\gcell(1,1)[] \gcell(1,2)[] \gcell(2,1)[] \gcell(2,2)[] \gcell(1,3)[] \gcell(3,1)[] 
\gcell(3,6)[] \gcell(4,6)[] \gcell(5,6)[]  \gcell(5,4)[] \gcell(5,5)[] 
\gcell(4,1)[] \bcirc(2,6) \bcirc(1,5)
\wcirc(2,2) \wcirc(3,3) \wcirc(4,4)
\wcirc(3,1) \wcirc(4,2) \wcirc(5,3)
\wcirc(4,1) \wcirc(5,2)
  \psgrid (0,0)(6,-5)
\end{pspicture}
\begin{pspicture}(-4,0)(6,-5)
\rput(-2,-2.5){$\Rightarrow$}
\gcell(1,1)[] \gcell(1,2)[] \gcell(2,1)[] \gcell(2,2)[] \gcell(1,3)[] \gcell(3,1)[] 
\gcell(3,6)[] \gcell(4,6)[] \gcell(5,6)[]  \gcell(5,4)[] \gcell(5,5)[] 
\gcell(4,1)[] \bcirc(2,6) \bcirc(1,5)
\wcirc(2,3) \wcirc(3,4) \wcirc(4,5)
\wcirc(3,1) \wcirc(4,2) \wcirc(5,3)
\wcirc(4,1) \wcirc(5,2)
  \psgrid (0,0)(6,-5)
\end{pspicture} 

\begin{pspicture}(0,0)(3,2)
\rput(3,1){$\downarrow \psi$}  
\end{pspicture}%
\begin{pspicture}(0,0)(23,2)
\rput(20.5,1){$||$}  
\end{pspicture}

\begin{pspicture}(0,0)(6,-5)
\gcell(1,1)[] \gcell(1,2)[] \gcell(2,1)[] \gcell(2,2)[] \gcell(1,3)[] \gcell(3,1)[] 
\gcell(3,6)[] \gcell(4,6)[] \gcell(5,6)[]  \gcell(5,4)[] \gcell(5,5)[] 
\gcell(4,1)[] \bcirc(2,6) \bcirc(1,5)
\wcirc(2,3) \wcirc(3,4) \wcirc(4,5)
\wcirc(3,2) \wcirc(4,3) \gcell(5,3)[]
\wcirc(4,2)  \gcell(5,2)[]
  \psgrid (0,0)(6,-5)
\end{pspicture}
\begin{pspicture}(-4,0)(6,-5)
\rput(-2,-2.5){$\Leftarrow$}
\gcell(1,1)[] \gcell(1,2)[] \gcell(2,1)[] \gcell(2,2)[] \gcell(1,3)[] \gcell(3,1)[] 
\gcell(3,6)[] \gcell(4,6)[] \gcell(5,6)[]  \gcell(5,4)[] \gcell(5,5)[] 
\gcell(4,1)[] \bcirc(2,6) \bcirc(1,5)
\wcirc(2,3) \wcirc(3,4) \wcirc(4,5)
\wcirc(3,2) \wcirc(4,3) \gcell(5,3)[]
\wcirc(4,1) \wcirc(5,2)
  \psgrid (0,0)(6,-5)
\end{pspicture}
\begin{pspicture}(-4,0)(6,-5)
\rput(-2,-2.5){$\Leftarrow$}
\gcell(1,1)[] \gcell(1,2)[] \gcell(2,1)[] \gcell(2,2)[] \gcell(1,3)[] \gcell(3,1)[] 
\gcell(3,6)[] \gcell(4,6)[] \gcell(5,6)[]  \gcell(5,4)[] \gcell(5,5)[] 
\gcell(4,1)[] \bcirc(2,6) \bcirc(1,5)
\wcirc(2,3) \wcirc(3,4) \wcirc(4,5)
\wcirc(3,1) \wcirc(4,2) \wcirc(5,3)
\wcirc(4,1) \wcirc(5,2)
  \psgrid (0,0)(6,-5)
\end{pspicture}
 \caption{An example of the involution $\psi$.}
  \label{fig:extended_inv}
\end{figure}

Now we are ready to prove Proposition~\ref{prop:dss}.

\begin{proof}[Proof of \eqref{eq:dss1}]
By Lemma~\ref{lem:general_dkim} we have
\begin{align*}
\sum_{i=0}^k \sum_{S\in \DSS(i,k-i)} \wt_{a,b}(S) &=
\sum_{i=0}^k \sum_{S\in \DSS(i,k-i;-, \emptyset,-,\emptyset)} \wt_{a,b}(S)\\
&=\sum_{u+v+2t=k} a^u b^v (-1)^t q^t 
\sum_{(\lambda,\emptyset,W,\emptyset) \in \DSS(u+t,v+t), |W|=t} q^{|\lambda|-\norm{W}}.
\end{align*}
It remains to show that for fixed nonnegative integers $u,v,t$ we have
\begin{equation}
  \label{eq:14}
\sum_{(\lambda,\emptyset,W,\emptyset) \in \DSS(u+t,v+t), |W|=t}
q^{|\lambda|-\norm{W}}
= q^{\binom{t}2} \qbinom{u+v+t}{u,v,t}.
\end{equation}
Suppose $(\lambda,\emptyset,W,\emptyset) \in \DSS(u+t,v+t)$ and $|W|=t$.  For
each white stripe we mark the row containing the leftmost white dot and chop off
the column containing the rightmost white dot, see
Figure~\ref{fig:dss_to_word}. Then the resulting partition is contained in a
$(u+t)\times v$ rectangle. Consider this partition as a lattice path from
$(0,0)$ to $(v,u+t)$ with three steps: an east step, a north step, and a marked
north step. Here a marked north step corresponds to a marked row of the
partition. Then we define the word $w=w_1w_2\dots w_{u+v+t}$ by $w_i=0,1$, or
$2$ if the $i$th step is a north step, a marked north step, or an east step
respectively. Then $w$ is a permutation of $u$ 0's, $t$ 1's and $v$ 2's.  It is
not hard to see that $(\lambda,\emptyset,W,\emptyset)\mapsto w$ is a bijection
from $\DSS(u+t,v+t)$ to the set $S(0^u,1^t,2^v)$ of permutations of $u$ 0's, $t$
1's and $v$ 2's such that
\[
|\lambda|-\norm{W} = \inv(w) + \binom{t}{2},
\]
where $\inv(w)$ is the number of pairs $(i,j)$ of integers such that $i<j$ and
$w_i>w_j$.
Since 
\[
\sum_{w\in S(0^u,1^t,2^v)} q^{\inv(w)} = \qbinom{u+v+t}{u,v,t},
\] 
we get \eqref{eq:14}. 

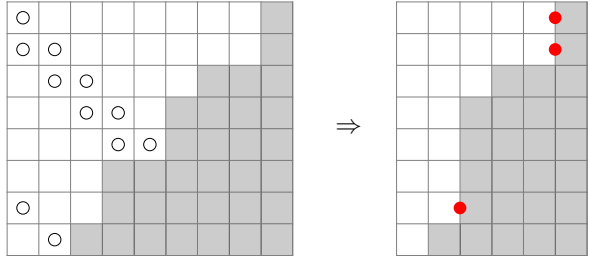
\begin{figure}
  \centering
\begin{pspicture}(0,0)(9,-8) 
\gcell(1,9)[] 
\gcell(2,9)[] 
\gcell(3,9)[] \gcell(3,8)[] \gcell(3,7)[] 
\gcell(4,9)[] \gcell(4,8)[] \gcell(4,7)[] \gcell(4,6)[] 
\gcell(5,9)[] \gcell(5,8)[] \gcell(5,7)[] \gcell(5,6)[] 
\gcell(6,9)[] \gcell(6,8)[] \gcell(6,7)[] \gcell(6,6)[] \gcell(6,5)[] \gcell(6,4)[] 
\gcell(7,9)[] \gcell(7,8)[] \gcell(7,7)[] \gcell(7,6)[] \gcell(7,5)[] \gcell(7,4)[] 
\gcell(8,9)[] \gcell(8,8)[] \gcell(8,7)[] \gcell(8,6)[] \gcell(8,5)[] \gcell(8,4)[] \gcell(8,3)[] 
\wcirc(7,1) \wcirc(8,2)  
\wcirc(1,1) \wcirc(2,2)  \wcirc(3,3)  \wcirc(4,4)  \wcirc(5,5)
\wcirc(2,1) \wcirc(3,2)  \wcirc(4,3)  \wcirc(5,4) 
  \psgrid (0,0)(9,-8)
\end{pspicture}
\begin{pspicture}(-3,0)(6,-8) 
\rput(-1.5,-4){$\Rightarrow$}
\gcell(1,6)[] 
\gcell(2,6)[] 
\gcell(3,6)[] \gcell(3,5)[] \gcell(3,4)[] 
\gcell(4,6)[] \gcell(4,5)[] \gcell(4,4)[] \gcell(4,3)[] 
\gcell(5,6)[] \gcell(5,5)[] \gcell(5,4)[] \gcell(5,3)[] 
\gcell(6,6)[] \gcell(6,5)[] \gcell(6,4)[] \gcell(6,3)[] 
\gcell(7,6)[] \gcell(7,5)[] \gcell(7,4)[] \gcell(7,3)[] 
\gcell(8,6)[] \gcell(8,5)[] \gcell(8,4)[] \gcell(8,3)[] \gcell(8,2)[] 
  \psgrid (0,0)(6,-8)
\pscircle*[linecolor=red](5,-.5){.2}
\pscircle*[linecolor=red](5,-1.5){.2}
\pscircle*[linecolor=red](2,-6.5){.2}
\end{pspicture}
\caption{An example of the intermediate step of the map from $\DSS(u+t,v+t)$ to
  $S(0^u,1^t,2^v)$. The word corresponding to this example is $2021000202211$.}
  \label{fig:dss_to_word}
\end{figure}

\end{proof}
 
\begin{proof}[Proof of \eqref{eq:dss2}]
For $S\in \DSS(m,n)$, let $S'$ be the doubly striped skew shape obtained by
rotating $S$ by an angle of $180^\circ$ and exchanging blacks stripes and white
stripes. The map $S\mapsto S'$ is clearly a bijection. If
$S=(\lambda,\mu,W,B)\in \DSS(m,n)$, then $S'=(\lambda',\mu',W',B')\in \DSS(m,n)$
satisfies $\lambda' = (n^m) - \mu$, $\mu' = (n^m) - \lambda$, $|W'|=|B|$,
$|B'|=|W|$, $\norm{W'}=\norm{B}$, and $\norm{B'}=\norm{W}$.
Thus $\wt_{a,b}(S) = \wt_{a,b}(S') (q/ab)^{|W|-|B|}$. In particular, we have
$\wt_{a,q/a}(S) = \wt_{a,q/a}(S')$. 

By Lemma~\ref{lem:general_dkim}, the map $S\mapsto S'$ and Lemma~\ref{lem:dkim},
we have
\begin{align*}
\sum_{i=0}^k \sum_{S\in \DSS(i,k-i)} \wt_{a,q/a}(S) &=
\sum_{i=0}^k \sum_{S\in \DSS(i,k-i;-, \emptyset,-,\emptyset)} \wt_{a,q/a}(S)\\
&=\sum_{i=0}^k \sum_{S'\in \DSS(i,k-i;((k-i)^i),-,\emptyset,-)} \wt_{a,q/a}(S')\\
&=\sum_{i=0}^k \sum_{S'\in \DSS(i,k-i;((k-i)^i),\emptyset,\emptyset,\emptyset)} \wt_{a,q/a}(S')
=\sum_{i=0}^k a^i (q/a)^{k-i} q^{i(k-i)},
\end{align*}
which finishes the proof.
\end{proof}

\section{Staircase tableaux}
\label{sec:staircase-tableaux}

In this section we review the second combinatorial model for the moments of 
the Askey-Wilson polynomials, called staircase tableaux. The staircase tableaux were first
introduced in \cite{Corteel2011} and further studied in \cite{CSSW}.  Using the
staircase tableaux we shall find the coefficient of the highest term in $2^n(abcd;q)_n
\mu_n(a,b,c,d;q)$ in Theorem~\ref{thm:highest}. 

A \emph{staircase tableau of size $n$} is a filling of the Young diagram of the
staircase partition $(n,n-1,\dots,1)$ with $\alpha,\beta,\gamma,\delta$ such
that
\begin{itemize}
\item every diagonal cell is nonempty, 
\item all cells above an $\alpha$ or $\gamma$ in the same column are empty,
\item all cells to the left of a $\beta$ or $\delta$ in the same row are empty.
\end{itemize}
Here a \emph{diagonal cell} is a cell in the $i$th row and $(n+1-i)$th column
for some $i\in \{1,2,\dots,n\}$. 

We denote by $\T(n)$ the set of staircase tableaux of size $n$.  

Each empty cell $s$ of $T\in\T(n)$ is labeled uniquely as follows. Here, for
brevity let $\RIGHT(s)$ be the first nonempty cell to the right of $s$ in the
same row, and $\BELOW(s)$ the first nonempty cell below $s$ in the
same column. 
\begin{itemize}
\item If $\RIGHT(s)$ has a $\beta$, then $s$ is labeled with $u$.
\item If $\RIGHT(s)$ has a $\delta$, then $s$ is labeled with $q$.
\item If $\RIGHT(s)$ has an $\alpha$ or $\gamma$, and $\BELOW(s)$ has an $\alpha$
  or $\delta$, then $s$ is labeled with $u$.
\item If $\RIGHT(s)$ has an $\alpha$ or $\gamma$, and $\BELOW(s)$ has an $\beta$
  or $\gamma$, then $s$ is labeled with $q$.
\end{itemize}
See Figure~\ref{fig:staircasetableau} for an example of a staircase tableau and
the labeling of its empty cells. 

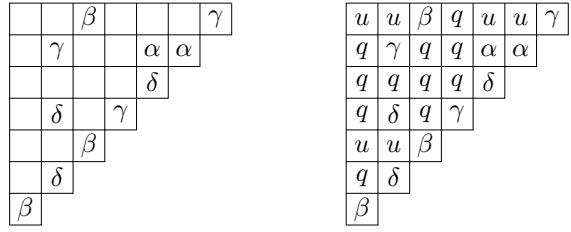
\begin{figure}
  \centering
\begin{pspicture}(0,0)(7,-7) \psline(7,0)(0,0)(0,-7) \cell(1,1)[]
  \cell(1,2)[] \cell(1,3)[$\beta$] \cell(1,4)[] \cell(1,5)[]
  \cell(1,6)[] \cell(1,7)[$\gamma$] \cell(2,1)[] \cell(2,2)[$\gamma$]
  \cell(2,3)[] \cell(2,4)[] \cell(2,5)[$\alpha$] \cell(2,6)[$\alpha$]
  \cell(3,1)[] \cell(3,2)[] \cell(3,3)[] \cell(3,4)[]
  \cell(3,5)[$\delta$] \cell(4,1)[] \cell(4,2)[$\delta$] \cell(4,3)[]
  \cell(4,4)[$\gamma$] \cell(5,1)[] \cell(5,2)[] \cell(5,3)[$\beta$]
  \cell(6,1)[] \cell(6,2)[$\delta$] \cell(7,1)[$\beta$]
\end{pspicture}
\qquad \qquad
\begin{pspicture}(0,0)(7,-7) \psline(7,0)(0,0)(0,-7) \cell(1,1)[$u$]
  \cell(1,2)[$u$] \cell(1,3)[$\beta$] \cell(1,4)[$q$] \cell(1,5)[$u$]
  \cell(1,6)[$u$] \cell(1,7)[$\gamma$] \cell(2,1)[$q$] \cell(2,2)[$\gamma$]
  \cell(2,3)[$q$] \cell(2,4)[$q$] \cell(2,5)[$\alpha$] \cell(2,6)[$\alpha$]
  \cell(3,1)[$q$] \cell(3,2)[$q$] \cell(3,3)[$q$] \cell(3,4)[$q$]
  \cell(3,5)[$\delta$] \cell(4,1)[$q$] \cell(4,2)[$\delta$] \cell(4,3)[$q$]
  \cell(4,4)[$\gamma$] \cell(5,1)[$u$] \cell(5,2)[$u$] \cell(5,3)[$\beta$]
  \cell(6,1)[$q$] \cell(6,2)[$\delta$] \cell(7,1)[$\beta$]
\end{pspicture}
  \caption{A staircase tableau and the labeling of its empty cells.}
  \label{fig:staircasetableau}
\end{figure}

For $T\in\T(n)$, we define $\blk(T)$ to be the number of $\alpha$'s and
$\delta$'s on the diagonal cells, and $A(T)$, $B(T)$, $C(T)$, $D(T)$, $E(T)$ to
be the number of $\alpha$'s, $\beta$'s, $\gamma$'s, $\delta$'s, empty cells
labeled with $q$ in $T$ respectively.  For example, if $T$ is the staircase
tableau in Figure~\ref{fig:staircasetableau}, we have $\blk(T)=3$, $A(T)=2$,
$B(T)=3$, $C(T)=3$, $D(T)=3$, and $E(T)=11$.

Let
\[
Z_n(y;\alpha,\beta,\gamma,\delta;q) 
=\sum_{T\in \T(n)} y^{\blk(T)} \alpha^{A(T)} \beta^{B(T)} \gamma^{C(T)}
\delta^{D(T)} q^{E(T)}.
\]

Corteel et al. \cite{CSSW} showed that
\begin{equation}
  \label{eq:CSSW}
\mu_n(a,b,c,d;q) = 
\frac{(1-q)^n}{2^n i^n \prod_{j=0}^{n-1}(\alpha\beta-\gamma\delta q^j)}
Z_n(-1;\alpha,\beta,\gamma,\delta;q),
\end{equation}
where 
\[
  \alpha = \frac{1-q}{(1+ai)(1+ci)}, \quad
\beta = \frac{1-q}{(1-bi)(1-di)}, \quad
\gamma = \frac{ac(1-q)}{(1+ai)(1+ci)}, \quad
\delta = \frac{bd(1-q)}{(1-bi)(1-di)}.
\]

Since 
\[
\alpha\beta-\gamma\delta q^j 
= \frac{(1-q)^2(1-abcdq^j)}{(1+ai)(1+ci)(1-bi)(1-di)}
\]
we can rewrite \eqref{eq:CSSW} as the next proposition.
\begin{prop} 
\label{prop:imagmom}
The Askey-Wilson moments satisfy
\begin{multline}
  2^n (abcd;q)_n \mu_n(a,b,c,d;q) = i^{-n} \sum_{T\in\T(n)} 
(-1)^{\blk(T)} (1-q)^{A(T)+B(T)+C(T)+D(T)-n} q^{E(T)} \\
\times  (ac)^{C(T)} (bd)^{D(T)}
\big((1+ai)(1+ci)\big)^{n-A(T)-C(T)} \big((1-bi)(1-di)\big)^{n-B(T)-D(T)}.
\nonumber
\end{multline}
\end{prop}

The highest degree term appearing in Proposition~\ref{prop:imagmom} is 
$a^n b^n c^n d^n q^{\binom n2}$. This term can be obtained when
  $A(T)+B(T)+C(T)+D(T)+E(T)=\binom{n+1}2$ and $A(T)=B(T)=0$, and the coefficient
  is
\[
i^{-n} (-1)^{\blk(T)}(-1)^{C(T)+D(T)-n} (-1)^{n-C(T)} (-1)^{n-D(T)} =i^{n}
(-1)^{\blk(T)}. 
\]
Thus 
\begin{equation}
  \label{eq:4}
\left[ a^n b^n c^n d^n q^{\binom n2} \right] 2^n(abcd;q)_n \mu_n(a,b,c,d;q)=
\sum_{T\in\CT(n)} i^{n} (-1)^{\blk(T)},
\end{equation}
where $\CT(n)$ is the set of $T\in\T(n)$ with $A(T)=B(T)=0$ and
$C(T)+D(T)+E(T)=\binom{n+1}2$.  We will see below that the tableaux in $\CT(n)$
are effectively the same as Catalan tableaux in \cite{Vi}.

\begin{figure}
  \centering
\begin{pspicture}(0,0)(9,-9) \psline(9,0)(0,0)(0,-9) \cell(1,1)[] \cell(1,2)[]
  \cell(1,3)[] \cell(1,4)[] \cell(1,5)[] \cell(1,6)[] \cell(1,7)[] \cell(1,8)[]
  \cell(1,9)[$\delta$] \cell(2,1)[] \cell(2,2)[] \cell(2,3)[] \cell(2,4)[]
  \cell(2,5)[] \cell(2,6)[] \cell(2,7)[$\delta$] \cell(2,8)[$\gamma$]
  \cell(3,1)[] \cell(3,2)[] \cell(3,3)[] \cell(3,4)[] \cell(3,5)[] \cell(3,6)[]
  \cell(3,7)[$\delta$] \cell(4,1)[$\delta$] \cell(4,2)[$\gamma$] \cell(4,3)[]
  \cell(4,4)[] \cell(4,5)[$\gamma$] \cell(4,6)[$\gamma$] \cell(5,1)[]
  \cell(5,2)[] \cell(5,3)[] \cell(5,4)[] \cell(5,5)[$\delta$] \cell(6,1)[]
  \cell(6,2)[$\delta$] \cell(6,3)[] \cell(6,4)[$\gamma$] \cell(7,1)[]
  \cell(7,2)[$\delta$] \cell(7,3)[$\gamma$] \cell(8,1)[] \cell(8,2)[$\delta$]
  \cell(9,1)[$\delta$]
\end{pspicture}
\qquad\qquad
\begin{pspicture}(0,0)(9,-9) \psline(5,0)(0,0)(0,-4) \cell(1,1)[] \cell(1,2)[]
  \cell(1,3)[] \cell(1,4)[\LTA] \cell(2,1)[\LTA] \cell(2,2)[\UPA] \cell(2,3)[\UPA] \cell(3,1)[]
  \cell(3,2)[\LTA] \cell(4,1)[] \cell(4,2)[\LTA]
\end{pspicture}
  \caption{A staircase tableau without $\alpha$'s and $\beta$'s and the corresponding alternative tableau.}
  \label{fig:alt_tab}
\end{figure}
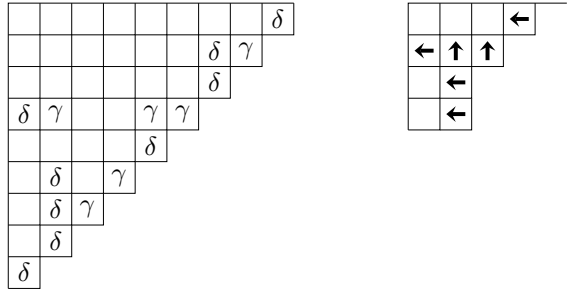

For the remainder of this section we will prove the following theorem.

\begin{thm}\label{thm:highest}
We have
\[
    \left[ a^n b^n c^n d^n q^{\binom n2} \right] 2^n(abcd;q)_n \mu_n(a,b,c,d;q)
 = \Cat\left(\frac n2\right),
\]
where $\Cat(n) =\frac{1}{n+1}\binom{2n}{n}$ if $n$ is a nonnegative integer, and
$\Cat(n)=0$ otherwise. 
\end{thm}

In order to prove the above theorem we recall alternative tableaux, which are in
simple bijection with permutation tableaux, see \cite{Nadeau,Viennota}.


An \emph{alternative tableau} is a filling of a Young diagram with possibly
empty rows and columns with up arrows and left arrows such that there is no
arrow pointing to another arrow. The \emph{size} of an alternative tableau is
the sum of the number of rows and columns.  A cell is called a \emph{free cell}
if there is no arrow pointing to this cell. An alternative tableau without free
cells is called a \emph{Catalan tableau}.  Viennot \cite{Viennot} showed that
the number of Catalan tableaux of size $n$ is be equal to $\Cat(n)$.

Suppose $T\in\T(n)$ satisfies $A(T)=B(T)=0$. Then nonempty cells of $T$ have
only $\gamma$ and $\delta$. We replace every $\gamma$ with an up arrow and every
$\delta$ with a left arrow. Then by definition every diagonal cell contains an
arrow and there is no arrow pointing to another arrow.  For each diagonal cell
we remove the column (resp.~row) containing it if it has an up arrow (resp.~left
arrow). Then the resulting object is an alternative tableau. If $\blk(T)=k$ then
the corresponding alternative tableau has $n-k$ rows and $k$ columns, see
Figure~\ref{fig:alt_tab}. If moreover $T$ satisfies
$C(T)+D(T)+E(T)=\binom{n+1}2$, thus $T\in\CT(n)$, then the corresponding
alternative tableau has no free cells, i.e., it is a Catalan tableau.  Thus we
obtain the following lemma.

\begin{lem}\label{lem:CT}
  The number of $T\in \CT(n)$ with $\blk(T)=k$ is equal to the number of Catalan
  tableaux of size $n$ with $n-k$ rows.
\end{lem}

The following lemma is proved by Burstein \cite[Theorem~3.1]{Burstein2007} using
the permutation tableaux  description.

\begin{lem}
\label{lem:burstein}
There is a bijection between the set of Catalan tableaux of size $n$ with $k$
rows and the set of noncrossing partitions of $\{1,2,\dots,n+1\}$ with $k$ blocks.
\end{lem}

Since the number of noncrossing partitions of $\{1,2,\dots,n\}$ with $k$ blocks
is the Narayana number $\nara(n,k) =
\frac{1}{n}\binom{n}{k} \binom{n}{k-1}$, by Lemmas~\ref{lem:CT} and
\ref{lem:burstein} we obtain 
\begin{equation}
  \label{eq:5}
\sum_{T\in\CT(n)} i^{n} (-1)^{\blk(T)} =
i^n \sum_{k=0}^{n} (-1)^k \nara(n+1,n+1-k).
\end{equation}

There is a formula for the alternating sum of Narayana numbers. 
\begin{lem}\cite[Proposition~2.2 corrected]{Bonin1993} \label{lem:altsum}
\[
\sum_{k=1}^n (-1)^k \nara(n,k) = (-1)^{\frac{n+1}2} \Cat\left(\frac{n-1}2\right).
\]
\end{lem}

Now we can prove Theorem~\ref{thm:highest}.

\begin{proof}[Proof of Theorem~\ref{thm:highest}]
By \eqref{eq:4} and \eqref{eq:5} we have  
\[
\left[ a^n b^n c^n d^n q^{\binom n2} \right] 2^n(abcd;q)_n \mu_n(a,b,c,d;q)=
i^n \sum_{k=1}^{n+1} (-1)^{n+1-k} \nara(n+1,k). 
\]
Then we are done by Lemma~\ref{lem:altsum}. 
\end{proof}

By analyzing the Catalan tableaux one can prove the following.  We will omit the
proofs.
\begin{thm}\label{thm:23}
We have
  \begin{align*}
    \left[ a^{n-1} b^n c^n d^n q^{\binom n2} \right] 2^n(abcd;q)_n \mu_n(a,b,c,d;q)
    &= -\Cat\left(\frac {n+1}2\right),\\
    \left[ a^{n-1} b^{n-1} c^n d^n q^{\binom n2} \right] 2^n(abcd;q)_n \mu_n(a,b,c,d;q)
    &= \Cat\left(\frac {n+2}2\right) - \Cat\left(\frac {n}2\right).
  \end{align*}
\end{thm}

\section{Matchings}

In this section we provide the third combinatorial approach to compute
$\mu_n(a,b,c,d;q),$ using matchings and the $q$-Hermite polynomials as in
\cite{ISV}.  We give a combinatorial expression for the non-normalized moments
in Theorem~\ref{thm:match}. We derive new expressions for $\mu_n(a,b,c,d;q)$ as
a rational function in Theorem~\ref{thm:main} and Theorem~\ref{thm:main2}, and a
new positivity result in Corollary~\ref{cor:flipcor}.

The approach of \cite{ISV} was to realize the Askey-Wilson integral 
\eqref{eq:AW_integral} as a generating function for an integral of a product of 
four $q$-Hermite polynomials
\begin{equation}
I_0(t_1,t_2,t_3,t_4)=\frac{(q)_\infty}{2\pi} \sum_{n_1,n_2,n_3,n_4\ge 0} 
\int_{0}^{\pi} 
\prod_{i=1}^4 \frac{H_{n_i}(\cos\theta|q)t_i^{n_i}}{(q)_{n_i}}  v(\cos\theta|q) d\theta.
\nonumber
\end{equation}
See \cite{ISV} for the definitions of $H_n(\cos\theta;q)$ and
$v(\cos\theta|q)$.

Put
\begin{equation}
\label{eq:I_n}
\begin{aligned}
I_n(t_1,t_2,t_3,t_4):= &\frac{(q)_\infty}{2\pi} \int_{0}^\pi (\cos\theta)^n 
w(\cos\theta,t_1,t_2,t_3,t_4;q) d\theta\\
=& \frac{(q)_\infty}{2\pi} \sum_{n_1,n_2,n_3,n_4\ge 0} 
\int_{0}^{\pi} (\cos\theta)^n
\prod_{i=1}^4 \frac{H_{n_i}(\cos\theta|q)t_i^{n_i}}{(q)_{n_i}}  v(\cos\theta|q) d\theta,
\end{aligned}
\end{equation}
so that
\begin{equation}
\label{inttomom}
I_n(a,b,c,d)=\mu_n(a,b,c,d;q)I_0(a,b,c,d).
\end{equation}
Upon rescaling, the Askey-Wilson integral \eqref{eq:AW_integral} followed from a
combinatorial evaluation of the integral of a product of four $q$-Hermite
polynomials, see \eqref{eq:7}.

The \emph{rescaled $q$-Hermite polynomials} $\RH_n(x|q)$ (see \cite{ISV}) are
defined in terms of the $q$-Hermite polynomials by
\[
\RH_n(x|q)= \frac{H_n(\frac{\sqrt{1-q}}2 x|q)}{(1-q)^{n/2}}.
\]
Let 
\begin{equation}
  \label{eq:tf}
\tf_n(n_1,n_2,n_3,n_4;q) = 
\frac{(q)_\infty}{2\pi}
\int_{-2/\sqrt{1-q}}^{2/\sqrt{1-q}} 
x^n\RH_{n_1}(x|q) \RH_{n_2}(x|q) \RH_{n_3}(x|q) \RH_{n_4}(x|q) \tv(x|q) dx,
\end{equation}
where 
\[
\tv(x|q) = \frac{v(\frac{\sqrt{1-q}}2 x|q) \cdot \frac{\sqrt{1-q}}2}{\sqrt{1-(1-q)x^2/4}}.
\]
By \eqref{eq:I_n} and \eqref{eq:tf} we have
\begin{equation}
  \label{eq:I_n2}
I_n(a,b,c,d) = \left(\frac{\sqrt{1-q}}2\right)^n \sum_{n_1,n_2,n_3,n_4\ge0} 
\frac{\ta^{n_1} \tb^{n_2} \tc^{n_3} \td^{n_4}}{[n_1]_q![n_2]_q![n_3]_q![n_4]_q!}
\tf_n(n_1,n_2,n_3,n_4;q),
\end{equation}
where $\ta = a/\sqrt{1-q}$, $\tb = b/\sqrt{1-q}$, $\tc = c/\sqrt{1-q}$, $\td =
d/\sqrt{1-q}$.  The rescaled $q$-Hermite polynomials have a close connection
with matchings. We need the following definitions.
 
A \emph{matching} is a set partition of $\{1,2,\dots,n\}$ in which every block
has size $1$ or $2$. A block of size $2$ is called an \emph{edge} and a block of
size $1$ is called a \emph{fixed point}.  A \emph{complete matching} is a
matching without fixed points. For a matching $\pi$ we define the \emph{crossing
  number} $\cro(\pi)$ to be the number of pairs $(e_1,e_2)$ of blocks
$e_1,e_2\in \pi$ such that $e_1=\{i_1,j_1\}, e_2=\{i_2,j_2\}$ with
$i_1<i_2<j_1<j_2$ or $e_1=\{i_1,j_1\}, e_2=\{j_2\}$ with $i_1<i_2<j_1$.  

For fixed integers $n_0,n_1,n_2,\dots,n_k$, let
$S_i=\{m_{i-1}+1,m_{i-1}+2,\dots,m_i\}$ for $i=0,1,2,\dots,k$, where $m_{-1}=0$
and $m_i = n_0+n_1+\dots+n_i$ for $i=0,1,\dots,k$.  We define
$\CM(n_0;n_1,n_2,\dots,n_k)$ to be the set of complete matchings $\pi$ on
$\bigcup_{i=0}^k S_i$ such that if an edge of $\pi$ is contained in $S_i$, then
$i=0$.


Ismail, Stanton and Viennot \cite[Theorem~3.2]{ISV} showed that
\begin{equation}
  \label{eq:7}
\tf_0(n_1,n_2,n_3,n_4;q) = \sum_{\sigma\in \CM(0;n_1,n_2,n_3,n_4)} q^{\cro(\sigma)}.  
\end{equation}


We have, using the same arguments as in \cite[Theorem~3.2]{ISV}, an 
analogous result for $\tf_n(n_1,n_2,n_3,n_4;q).$

\begin{thm}
\label{thm:match}
We have
\[
\tf_n(n_1,n_2,n_3,n_4;q) = \sum_{\sigma\in \CM(n;n_1,n_2,n_3,n_4)} q^{\cro(\sigma)}.
\]
\end{thm}

Due to \eqref{inttomom} and \eqref{eq:I_n2}, one may use Theorem~\ref{thm:match}
to give another explicit formula for $\mu_n(a,b,c,d;q).$

We can use Theorem~\ref{thm:match} to find new explicit formulas for the
moments, and also explain the mixed binomial and $q$-binomial coefficients. The
reason is that the generating function for a crossing number of matchings always
has such a formula.

Let $\matching(n,m)$ denote the set of matchings on $\{1,2,\dots,n\}$ with $m$
fixed points. Note that $\matching(n,m)=\emptyset$ unless $n\equiv m \mod 2$.
Let 
\[
P(n,m)=\sum_{\pi\in\matching(n,m)} q^{\cro(\pi)}, \qquad 
\overline P(n,m) = (1-q)^{(n-m)/2} P(n,m).
\]
\JV. \cite[Proposition~5.1]{JV_rook} showed the following (see also
\cite[Proposition~15]{Cigler2011}): if $n\equiv m \mod 2$, we have
\begin{equation}
\label{eq:OP}  
\overline P(n,m) = 
\sum_{k\ge0} \left( \binom{n}{\frac{n-k}2} - \binom{n}{\frac{n-k}2-1}\right)
(-1)^{(k-m)/2} q^{\binom{(k-m)/2+1}2} 
\qbinom{\frac{k+m}2}{\frac{k-m}2}.
\end{equation}



\begin{thm}\label{thm:main}
We have
\[
2^n\mu_n(a,b,c,d;q) =  \sum_{\abcd,\ge0}\abcdpower \op(n,\abcd+)
\qbinom{\abcd+}{\abcd,} \frac{(ad)_{\bg}(ac)_{\beta}(bd)_{\gamma}}{(abcd)_\bg}, 
\]
where $\overline{P}$ is given by \eqref{eq:OP}.
\end{thm}
\begin{proof}
  Consider a complete matching $\sigma\in \CM(n;n_1,n_2,n_3,n_4)$.  Let
  $S_0=\{1,\dots,n\}$ and $S_i=\{m_{i-1}+1,m_{i-1}+2,\dots,m_i\}$ for
  $i=1,2,3,4$, where $m_0=n$ and $m_i = n+n_1+\dots+n_i$.  Let $\alpha$
  (respectively $\beta,\gamma,\delta$) be the number of edges between $S_0$ and
  $S_1$ (respectively $S_2,S_3,S_4$) in $\sigma$.  For
  $i,j\in\{1,2,3,4\}$, let $n_{ij}$ be the number of edges between the blocks
  $S_i$ and $S_j$ in $\sigma$.  It is straightforward to check that the sum
  of $q^{\cro(\sigma)}$ for all $\sigma$ having the above numbers is equal to
  \begin{align*}
& P(n,\abcd+) \qbinom{\abcd+}{\abcd,} 
q^{\beta n_{13} +\beta n_{14} + \gamma n_{14} + \gamma n_{24} + n_{13}n_{24}}\\
& \times \qbinom{n_1}{\alpha,n_{12},n_{13},n_{14}}
\qbinom{n_2}{n_{21},\beta,n_{23},n_{24}}
\qbinom{n_3}{n_{31},n_{32},\gamma,n_{34}}
\qbinom{n_4}{n_{41},n_{42},n_{43},\delta}\\
& \times [\alpha]_q![\beta]_q![\gamma]_q![\delta]_q! 
[n_{12}]_q![n_{13}]_q![n_{14}]_q![n_{23}]_q![n_{24}]_q![n_{34}]_q!.
  \end{align*}
Thus by \eqref{eq:I_n} we have
\begin{align*}
I_n &= \left(\frac{\sqrt{1-q}}2\right)^n 
\sum_{\abcd,\ge0}
\ta^{\alpha} \tb^{\beta} \tc^{\gamma} \td^{\delta}
P(n,\abcd+) \qbinom{\abcd+}{\abcd,} \\
&\qquad \times
\sum_{n_{12}\ge0} \frac{\left(\ta\tb\right)^{n_{12}}}{[n_{12}]_q!}
\sum_{n_{14}\ge0} \frac{\left(\ta\td q^{\beta+\gamma}\right)^{n_{14}}}{[n_{14}]_q!}
\sum_{n_{23}\ge0} \frac{\left(\tb\tc\right)^{n_{23}}}{[n_{23}]_q!}
\sum_{n_{34}\ge0} \frac{\left(\tc\td\right)^{n_{34}}}{[n_{34}]_q!}\\ 
&\qquad \times
\sum_{n_{13},n_{24}\ge0} 
\frac{\left(\ta\tc q^{\beta}\right)^{n_{13}}}{[n_{13}]_q!}
\frac{\left(\tb\td q^{\gamma}\right)^{n_{24}}}{[n_{24}]_q!}\\
&= \frac{1}{2^n} \sum_{\abcd,\ge0}
a^{\alpha} b^{\beta} c^{\gamma} d^{\delta}
\op(n,\abcd+) \qbinom{\abcd+}{\abcd,} \\
&\qquad \times
\frac{(abcdq^{\beta+\gamma})_\infty}
{(ab)_\infty (adq^{\beta+\gamma})_\infty (bc)_\infty (cd)_\infty 
(acq^{\beta})_\infty (bdq^{\gamma})_\infty}.
\end{align*}
Here we used the following identities:
\[
\frac{(x/(1-q))^n}{[n]_q!} =\frac{x^n}{(q)_n}, \qquad
\sum_{n\ge0} \frac{x^n}{(q)_n} = (x)_n, \qquad
\sum_{n,m\ge0} \frac{q^{nm}x^ny^m}{(q)_m(q)_n} = 
\frac{(xy)_\infty}{(x)_\infty (y)_\infty}.
\]
Since \eqref{inttomom} is $\mu_n(a,b,c,d;q)=I_n/I_0$, 
we are done by \eqref{eq:AW_integral}.
\end{proof}

\begin{cor}
\[
2^n\mu_n(a,b,c,0;q) = \sum_{\alpha,\beta,\gamma\ge0}
a^{\alpha} b^{\beta} c^{\gamma} 
\op(n, \alpha+\beta+\gamma) 
\qbinom{\alpha+\beta+\gamma}{\alpha,\beta,\gamma}
(ac)_{\beta}.
\]
\end{cor}
\begin{proof}
  Since $\mu_n(a,b,c,d;q)$ is symmetric in $a,b,c,d$, it is sufficient to find
  $\mu_n(a,b,0,d;q)$.  If $c=0$ in Theorem~\ref{thm:main}, then nonzero terms
  have $\gamma=0$. Thus we obtain the desired identity with $c$ and $\gamma$
  replaced with $d$ and $\delta$ respectively.
\end{proof}

\begin{cor}
\[
2^n\mu_n(a,b,0,0;q) = \sum_{\alpha,\beta\ge0}
a^{\alpha} b^{\beta}  \op(n, \alpha+\beta) 
\qbinom{\alpha+\beta}{\alpha}.
\]
\end{cor}
\begin{proof}
  This is an immediate consequence of Theorem~\ref{thm:main} when $c=d=0$. 
\end{proof}

Note that the moment $2^n\mu_n(a,b,0,0;q)$ is not a polynomial with positive
coefficients. However, if we use a factor that is a power of $1-q$, we can
express this moment using polynomials with positive coefficients as follows.

\begin{prop}
\label{prop:partmatch}
We have
\[
2^n\mu_n(a,b,0,0;q) = \sum_{u,v\ge0} a^u b^v \qbinom{u+v}{u}
(1-q)^{(n-u-v)/2}\sum_{\pi\in\matching(n,u+v)} q^{\cro(\pi)}.
\]
\end{prop}

Another corollary is a simplified version of the $q=0$ result in 
\cite[Theorem 3.7]{CSSW}.

\begin{cor}
\[
2^n\mu_n(a,b,c,d;0) = 
\sum_{k=0}^{n} \left(\binom{n}{\frac{n-k}2}-\binom{n}{\frac{n-k}2-1}\right)
\sum_{\abcd+=k}\abcdpower \Phi(\beta,\gamma),
\]
where
\[
\Phi(\beta,\gamma) = \left\{
  \begin{array}{ll}
    \frac{(1-ac)(1-bd)(1-ad)}{1-abcd}  & \mbox{if $\beta\ne0$ and $\gamma\ne0$,}\\
\frac{(1-bd)(1-ad)}{1-abcd} & \mbox{if $\beta=0$ and $\gamma\ne0$,}\\
\frac{(1-ac)(1-ad)}{1-abcd}  & \mbox{if $\beta\ne0$ and $\gamma=0$,}\\
   1 & \mbox{if $\beta=\gamma=0$.}
  \end{array}\right.
\]
\end{cor}

Finally, we shall use Theorem~\ref{thm:main} to give an yet another 
formula, Theorem~\ref{thm:main2},  for the moments. We shall need the $q$-analogue of Saalschutz's theorem, for 
which Zeilberbger  \cite{Zeilberger1987} gave a bijective proof.
\begin{lem}\label{lem:Zeilberger}
\[
\qbinom{a+b+k}{a}\qbinom{b+c+k}{b}\qbinom{c+a+k}{c}  
=\sum_{j\ge0} q^{j(j+k)}
\frac{[a+b+c+k-j]_q!}{[a-j]_q![b-j]_q![c-j]_q![j]_q![k+j]_q!}.
\]
\end{lem}

\begin{thm}\label{thm:main2}
We have
\begin{multline*}
2^n\mu_n(a,b,c,d;q) =
\sum_{k=0}^{n} \left(\binom{n}{\frac{n-k}2}-\binom{n}{\frac{n-k}2-1}\right)
\sum_{\abcd+ +2t=k}\abcdpower \frac{(ac)_{\beta}(bd)_{\gamma}}{(abcd)_\bg}\\
\times(-1)^{t} q^{\binom{t+1}2} 
\qbinom{\alpha+\beta+\gamma+t}{\alpha}
\qbinom{\beta+\gamma+\delta+t}{\beta,\gamma,\delta+t}
\qbinom{\delta+\alpha+t}{\delta},
\end{multline*}
in the second sum, $\abcd,\ge0$ and $-k\le t\le k/2$.
\end{thm}
\begin{proof}
  In Theorem~\ref{thm:main}, if we expand $(ad)_{\bg}$ using the $q$-binomial theorem
\[
(ad)_{\bg} = \sum_{j\ge0} (-1)^j q^{\binom j2} \qbinom{\bg}{j} (ad)^j
\]
and change the indices $\alpha\mapsto \alpha-j$ and $\delta\mapsto \delta-j$, we
get
\[
2^n\mu_n(a,b,c,d;q) = \sum_{\abcd,,\ge0}\abcdpower 
\frac{(ac)_{\beta}(bd)_{\gamma}}{(abcd)_\bg} \qbinom{\bg}{\beta} X,
\]
where
  \begin{align*}
    X & = \sum_{j\ge0} (-1)^j q^{\binom j2} \op(n,\abcd+-2j)
    \frac{[\abcd+-2j]_q!}{[\alpha-j]_q! [\bg-j]_q! [\delta-j]_q! [j]_q!}.
  \end{align*}
We have
\[
\op(n,\abcd+-2j)
= \sum_{k\ge0}\left( \binom{n}{\frac{n-k}2} -
      \binom{n}{\frac{n-k}2-1}\right) 
(-1)^{t+j} q^{\binom{t+j+1}2} \qbinom{k-t-j}{t+j},
\]
where $t=(k-\abcd-)/2$.
Thus
\[
X = \sum_{k\ge0}\left( \binom{n}{\frac{n-k}2} -
      \binom{n}{\frac{n-k}2-1}\right) (-1)^t q^{\binom{t+1}2}
\sum_{j\ge0} q^{j(j+t)}
    \frac{[\abcd++t-j]_q!}{[\alpha-j]_q! [\bg-j]_q! [\delta-j]_q! [j]_q! [t+j]_q!}.
\]
By Lemma~\ref{lem:Zeilberger}, we can evaluate the $j$-sum and we get the
desired formula. 
\end{proof}

In Theorem~\ref{thm:main2}, if $c=0$, then $\gamma=0$ in the sum, and we obtain
Theorem~\ref{thm:AWWd=0}. We also have another corollary which includes a
 positivity result.

\begin{cor}\label{cor:flipcor}
We have
\[
2^n\mu_n(a,b,q/a,q/b;q) =
\sum_{k=0}^{n} \left(\binom{n}{\frac{n-k}2}-\binom{n}{\frac{n-k}2-1}\right)
\frac{1}{[k+1]_q} 
\sum_{\substack{|A|+|B|\le k \\A+B\equiv k \mod2}} a^A b^B q^{\frac{k-A-B}2},
\]  
where the second sum is over all integers $A$ and $B$ such that
$|A|+|B|\le k$ and $A+B\equiv k \mod2$.
Thus $[n+1]_q! 2^n\mu_n(a,b,q/a,q/b;q)$ is a Laurent polynomial in $a$ and $b$
whose coefficients are positive polynomials in $q$. 
\end{cor}
\begin{proof}
  When $c=q/a$ and $d=q/b$, we have
  $\frac{(ac)_{\beta}(bd)_{\gamma}}{(abcd)_\bg}\qbinom{\bg}{\beta}
  =\frac{1}{[\bg+1]_q}$. Thus by Theorem~\ref{thm:main2}, letting
  $A=\alpha-\gamma$ and $B=\beta-\delta$, we get
\[
2^n\mu_n(a,b,q/a,q/b;q) =\sum_{k=0}^{n} 
\left(\binom{n}{\frac{n-k}2}-\binom{n}{\frac{n-k}2-1}\right)
\sum_{\substack{|A|+B\le k\\-A+|B|\le k \\A+B\equiv k \mod2}}
a^A b^B q^{\frac{k-A-B}2} \cdot X(A,B),
\]
where
\begin{equation}
\begin{aligned}
X(A,B)=&\sum_{\beta=\max(B,0)}^{\frac{k+A+B}2}
\frac{1}{[\beta+\frac{k-A-B}2+1]_q} \qbinom{\frac{k+A-B}2}{\beta-B}\\
&\times \sum_{\alpha=\max(A,0)}^{\frac{k+A-B}2} (-1)^{t} q^{\binom{t}2} 
\qbinom{\beta+\frac{k-A-B}2+1}{\beta+\alpha-A+1}
\qbinom{\frac{k-A+B}2+\alpha}{\alpha},
\end{aligned}
\nonumber
\end{equation}
$t=\frac{k+A+B}2 - \alpha-\beta$.  Note that the ranges of indices are
determined by the $q$-binomial coefficients in the sums of $X(A,B)$. We will
consider the two cases: $A<0$ and $A\ge0$.

\textsc{Case 1}: $A<0$. One can evaluate the $\alpha$-sum of $X(A,B)$ using the
$q$-Vandermonde \cite[(II.7)]{GR} 
\[
X(A,B) =\sum_{\beta=\max(B,0)}^{\frac{k+A+B}2}
\frac{1}{[\beta+\frac{k-A-B}2+1]_q} \qbinom{\frac{k+A-B}2}{\beta-B}
(-1)^{\frac{k+A+B}2-\beta} q^{\binom{\frac{k+A+B}2-\beta}2}
\frac{(q^{\beta-\frac{k+A+B}2+1})_{\frac{k+A-B}2}}
{(q^{\beta-A+2})_{\frac{k+A-B}2}}.
\]
Since $B\le\beta\le\frac{k+A+B}2$, we have $-\frac{k+A-B}2\le
\beta-\frac{k+A+B}2\le0$. Thus $(q^{\beta-\frac{k+A+B}2+1})_{\frac{k+A-B}2} =0$
unless $\beta=\frac{k+A+B}2$, which can occur if and only if $k+A-B\ge0$ and
$k+A+B\ge0$. Thus when $A<0$ we have $X(A,B)=\frac{1}{[k+1]_q}$ if $|A|+|B|\le
k$, and $X(A,B)=0$ otherwise.

\textsc{Case 2}: $A\ge0$. Note that in this case we have $A+|B|=|A|+|B|\le k$. Then in
the $\alpha$-sum we have $A\le \alpha\le \frac{k+A-B}2$. By changing the index
$\alpha \mapsto \frac{k+A-B}2-\alpha$, and using the other type of
$q$-Vandermonde \cite[(II.6)]{GR}, one can check that the $\alpha$-sum and of
$X(A,B)$ is equal to
\[
(-1)^{\beta-B} q^{\binom{\beta-B+1}2 -(k+1)\frac{k+A-B}2} 
\qbinom{k}{\frac{k+A-B}2} 
\frac{(q^{\beta-\frac{k+A+B}2+1})_{\frac{k+A-B}2}}{(q^{-k})_{\frac{k+A-B}2}},
\] 
which is zero unless $\beta=\frac{k+A+B}2$. Since $|A|+|B|\le k$, there is
$\beta=\frac{k+A+B}2$ in the $\beta$-sum. In fact, when $\beta=\frac{k+A+B}2$
one can check that the above term is $1$.  Thus when $A\ge0$ we have
$X(A,B)=\frac{1}{[k+1]_q}$.

In both cases, we have $X(A,B)=\frac{1}{[k+1]_q}$ when $|A|+|B|\le k$ and 
$X(A,B)=0$ otherwise. Thus we obtain the identity.
\end{proof}

 One can show that the sum of the coefficients in Corollary~\ref{cor:flipcor} is   $2^n(n+1)!.$

\begin{conj}
  $2^n [n+i+j-1]_q! \mu_n(a,b,q^i/a,q^j/b;q)$ is a Laurent polynomial in $a,b$
  and polynomial in $q$ with nonnegative coefficients.
\end{conj}

Next we give an alternative formula for $\mu_{n}(a,b,-b,-a;q)$ using
Theorem~\ref{thm:main}. Note that $\mu_{n}(a,b,-b,-a;q)=0$ unless $n$ is even.
We need the following lemma, \cite[(II.9)]{GR} 

\begin{lem}\label{lem:andrews}
 \[
\sum_{i=0}^n (-1)^i \qbinom ni (x)_{n-i} (x)_i
=\left\{
  \begin{array}{ll}
    (q;q^2)_{n/2}  (x^2;q^2)_{n/2}, & \mbox{if $n$ is even,}\\
    0, & \mbox{if $n$ is odd.}
  \end{array}
\right.
\]
\end{lem}

\begin{thm}
\[
4^n\mu_{2n}(a,b,-b,-a;q) = \sum_{\alpha,\beta} a^{2\alpha} b^{2\beta}
\op(2n,2\alpha+2\beta) \qbinom{2\alpha+2\beta}{2\alpha}
\frac{(q;q^2)_\alpha (q;q^2)_\beta (-a^2;q)_{2\beta}}{(qa^2b^2;q^2)_\beta}
\]  
\end{thm}
\begin{proof}
  Letting $c=-b$, $d=-a$, and changing the indices $\alpha+\delta\mapsto
  \alpha$, $\bg\mapsto \beta$ in Theorem~\ref{thm:main}, we get
  \begin{align*}
4^n\mu_{2n}(a,b,-b,-a;q) &= \sum_{\alpha,\beta\ge0} a^{\alpha} b^{\beta} 
\op(2n,\alpha+\beta) \qbinom{\alpha+\beta}{\alpha}
\frac{(-a^2)_{\beta}}{(a^2b^2)_{\beta}}\\
&\qquad \times
\sum_{\gamma\ge0} (-1)^{\gamma} \qbinom{\beta}{\gamma}
(-ab)_{\beta-\gamma} (-ab)_{\gamma}
\sum_{\delta\ge0} (-1)^{\delta} \qbinom{\alpha}{\delta}.
  \end{align*}
By Lemma~\ref{lem:andrews}, we get the theorem.
\end{proof}


If we use the techniques of this section on the $q^2$-Hermite polynomials 
$H_n(x|q^2)$, then we can find another alternative theorem, whose proof we omit.

\begin{thm}
\[
4^n \mu_{2n}(a,b,-a,-b;q) = \sum_{m=0}^n \overline P(2n,2m)
\frac{(q;q^2)_m}{(qa^2b^2;q^2)_m}
\sum_{\alpha+\beta+i=m} a^{2\alpha+2i} b^{2\beta+2i}
(-q;q)_i q^{\binom i2} \Qbinom{m}{\alpha,\beta,i}{q^2}
\]
\end{thm}

\section{Connections with other results}
\label{sec:conn-with-other}


Recently the problem of finding formulas for moments of orthogonal polynomials has
gained some attention.  In this section we will rederive several moment formulas
using Corollary~\ref{cor:AWWc=d=0} and Corollaries~\ref{cor:bb=q/a} and
\ref{cor:bb=-a}.  We will use the following fact.

\begin{lem}\label{lem:rescaling}
If $\mu_n$ is the $\nth$ moment of the orthogonal polynomial $P_n(x)$ defined by
\[
P_{n+1}(x) = (x-b_n) P_n(x) -\lambda_n P_{n-1}(x),
\]
and $\mu'_n$ is the $\nth$ moment of the orthogonal polynomials $P'_n(x)$
defined by 
\[
P'_{n+1}(x) = (x-c(d+b_n)) P'_n(x) -c^2\lambda_n P'_{n-1}(x), 
\]
then
\[
\mu'_n = c^n \sum_{m=0}^n d^{n-m} \binom nm \mu_m.
\]
\end{lem}

\subsection{The $q$-Laguerre polynomials}
The $q$-Laguerre polynomials are defined by $b_n=[n]_q+y[n+1]_q$ and $\lambda_n
= y[n]_q^2$.  Let $\nu_n$ denote the $\nth$ moment of the $q$-Laguerre
polynomials.  We refer the reader to \cite[Proposition 1]{CJVPR} for many
interesting interpretations of $\nu_n$ involving permutations, permutation
tableaux, matrix ansatz, etc.

\JV. \cite{JV_rook} and Corteel et al. \cite{CJVPR} showed that
\begin{equation}
  \label{eq:corteel}
\nu_n = \frac{1}{(1-q)^n} \sum_{k=0}^n  \sum_{j=0}^{n-k} y^j 
  \left( \binom{n}{j}\binom{n}{j+k} - \binom{n}{j-1}\binom{n}{j+k+1} \right)
\sum_{i=0}^k (-1)^k y^i q^{i(k+1-i)}.
\end{equation}

Using Lemma~\ref{lem:rescaling} we have
\[
\nu_n = y^{n/2} \sum_{m=0}^n (y^{1/2} +y^{-1/2})^{n-m} \binom{n}{m} 
2^m \mu_m( -qy^{1/2},-y^{-1/2} , 0,0;q).
\]
By using Corollary~\ref{cor:bb=q/a} for $2^m \mu_m( -qy^{1/2},-y^{-1/2} , 0,0;q)$, 
the binomial theorem for $(y^{1/2} +y^{-1/2})^{n-m}$, and Vandermonde's theorem 
to sum the $m$-sum, we obtain \eqref{eq:corteel}. 

Using the rank generating function for the totally nonnegative Grassmannian
cells, Williams \cite[Corollary 5.3]{Williams2005} showed another formula for $\nu_n$:
\begin{equation}
  \label{eq:williams}
\nu_n = \sum_{i=0}^n y^i \sum_{j=0}^{i} (-1)^j \qint{i-j}^n q^{i(j-i)}
\left( \binom{n}{j} q^{i-j} + \binom{n}{j-1} \right).
\end{equation}

We now show that the two formulas \eqref{eq:corteel} and
\eqref{eq:williams} are equivalent under simple manipulation.  In other words,
we will prove the following identity:
\begin{multline}
  \label{eq:16}
(1-q)^n \sum_{i=0}^n \sum_{j=0}^{i} (-1)^j \qint{i-j}^n y^i q^{i(j-i)}
\left( \binom{n}{j} q^{i-j} + \binom{n}{j-1} \right)\\
= \sum_{k=0}^n 
\sum_{j=0}^{n-k} y^j 
  \left( \binom{n}{j}\binom{n}{j+k} - \binom{n}{j-1}\binom{n}{j+k+1} \right)
\sum_{i=0}^k (-1)^k y^i q^{i(k+1-i)}.
\end{multline}

By the binomial expansion for $(1-q)^n \qint{i-j}^n = (1-q^{i-j})^n$, the left
hand side of \eqref{eq:16} is equal to
\begin{multline*}
\sum_{i=0}^n \sum_{j=0}^{i} (-1)^j 
y^i  q^{i(j-i)} \sum_{k=0}^n \binom{n}{k} (-q^{i-j})^k 
\left( \binom{n}{j} q^{i-j} + \binom{n}{j-1} \right)\\
= \sum_{i=0}^n \sum_{j=0}^{i} \sum_{k=0}^{n+1}  (-1)^{j+k}
y^i q^{(i-j)(k-i)}  \left(
\binom{n}{j-1}\binom{n}{k}  - \binom{n}{j}\binom{n}{k-1}  \right).
\end{multline*}
The summand is anti-symmetric in $k$ and $j$. Thus we can let $k$ range
from $i+1$ to $n+1$.  Replacing $k\mapsto k+j+1$ and $i\mapsto i+j$ and changing
the order of sums we get the right hand side of \eqref{eq:16}.

\subsection{\JV.'s formula}

\JV. \cite{JV_PASEP} defined a partition function $Z_n(a,b,y,q)$ and showed that
$(1-q)^n Z_n(a,b,y,q)$ is equal to the $\nth$ moment defined by $b_n = 1+y
+(a+by)q^n$ and $\lambda_n = y(1-q^i)(1-abq^{i-1})$. 
He found the following formula for $Z_n(a,b,y,q)$
\begin{multline}\label{eq:JV-Zn}
Z_n(a,b,y,q)  =  \frac1{(1-q)^n} \sum_{k=0}^n 
\sum_{i=0}^{\flr{\frac{n-k}2}} (-y)^i q^{\binom{i+1}{2}} \qbinom{k+i}{i}\\
\times \sum_{j=0}^{n-k-2i} y^j \left(
\tbinom{n}{j}\tbinom{n}{k+2i+j}-\tbinom{n}{j-1}\tbinom{n}{k+2i+j+1} 
\right) \sum_{r=0}^k \qbinom{k}{r} a^r (yb)^{k-r}.
\end{multline}

 By the Rescaling Lemma~\ref{lem:rescaling}, we get
\[
(1-q)^n Z_n(a,b,y,q) = y^{n/2}
\sum_{m=0}^n (y^{1/2} +y^{-1/2})^{n-m} \binom{n}{m} 
2^m \mu_m(ay^{-1/2} , by^{1/2}, 0,0;q). 
\]
Thus we get in the same manner as \eqref{eq:corteel}

  \begin{multline*}
(1-q)^n Z_n(a,b,y,q) = \sum_{k=0}^n \sum_{j=0}^{n-k} y^j 
  \left( \binom{n}{j}\binom{n}{j+k} - \binom{n}{j-1}\binom{n}{j+k+1} \right)\\
\times
\sum_{u+v+2t=k} a^u b^v y^{v+t}(-1)^{t} q^{\binom{t+1}2} \qbinom{u+v+t}{u,v,t}.
  \end{multline*}

  It is easy to check that the above formula is equivalent to \eqref{eq:JV-Zn}.

\subsection{\JV. and Rubey's formula}

\JV. and Rubey \cite{JVR} considered the $\nth$ moment $R_n(a,b,c,d;q)$ of the
orthogonal polynomials defined by $b_n=d+(a+b)q^n$ and $\lambda_n =
(1-q^n)(c-abq^{n-1})$ and showed that
\begin{equation}
  \label{eq:JVRmu}
R_n(a,b,c,d;q) = \sum_{k=0}^n \frac{k+1}{n+1} 
\sum_{\ell=0}^{n-k} \binom{n+1}{\ell} \binom{\ell}{2\ell-n+k}
c^{n-k-\ell} d^{2\ell-n+k} M_k^*(a,b,c;q),  
\end{equation}
where
\[
\sum_{k\geq0} M_k^*(a,b,c;q) z^k = \frac{1}{1-at} \twophione{cb^{-1}qz, q}{aqz}{q,bz}.  
\]
Using \eqref{eq:JVRmu} \JV. and Rubey \cite{JVR} derive moment formulas for
various orthogonal polynomials. 

By the Rescaling Lemma~\ref{lem:rescaling} we have
\[
R_n(a,b,c,d;q)  = c^{n/2}
\sum_{m=0}^n (dc^{-1/2})^{n-m} \binom{n}{m} 
\mu_m(ac^{-1/2} , bc^{-1/2}, 0,0;q). 
\]
Thus we get 
\begin{equation}
  \label{eq:10}
R_n(a,b,c,d;q) = \sum_{k=0}^n \sum_{m=k}^n c^{\frac{m-k}2} d^{n-m} \tbinom{n}{m} 
\left( \tbinom{m}{\frac{m-k}2}-\tbinom{m}{\frac{m-k}2-1}\right)
\sum_{u+v+2t=k} a^u b^v c^t (-1)^{t} q^{\binom{t+1}2} \qbinom{u+v+t}{u,v,t}.
\end{equation}

It is easy to check using $q$-binomial theorems that \eqref{eq:10}
is equivalent to \eqref{eq:JVRmu}. 

\subsection{The $(t,q)$-Euler numbers}

Kim \cite{Kim_qEuler} defined the $(t,q)$-Euler numbers $E_n(t,q)$ by
\[ \sum_{n\geq0} E_{n}(t,q) z^n = \cfrac{1}{
    1- \cfrac{[1]_q [1]_{t,q} z}{
       1- \cfrac{[2]_q [2]_{t,q} z}{\cdots}}},
\]
and showed that
\[ E_n(t,q) = \frac{1}{(1-q)^{2n}} \sum_{k=0}^n \left( \binom{2n}{n-k} -
  \binom{2n}{n-k-1} \right) t^k q^{k(k+1)} T_k(t^{-1}, q^{-1}),\] 
where $[n]_{t,q} = (1-tq^n)/(1-q)$ and 
\begin{equation}
  \label{eq:Tk}
T_k(t,q)=\sum_{j=0}^k \sum_{i=0}^j (-1)^{j+i}
t^{2i} q^{j^2+i^2+i} \Qbinom{k-j}{i}{q^2}
\left( \Qbinom{k-i}{j-i}{q^2}
+ t \Qbinom{k-i-1}{j-i-1}{q^2}
\right).
\end{equation}

Note that we have $(1-q)^{2n} E_n(t,q) = 2^{2n} \mu_{2n}(\sqrt{-qt},
-\sqrt{-qt},0,0;q)$.  Thus by Corollary~\ref{cor:bb=-a} implies 

\begin{equation}\label{prop:Entq}
E_n(t,q) = \frac{1}{(1-q)^{2n}}
\sum_{k=0}^{n} \left( \binom{2n}{n-k}-\binom{n}{n-k-1}\right)
\sum_{i=0}^k (-1)^k q^{\binom{i+1}2} (q;q^2)_{k-i} (qt)^{k-i} \qbinom{2k-i}{i}. 
\end{equation}

We now show that the above two formulas for $E_n(t,q)$ are equivalent. 
It is sufficient to show that
\begin{equation}
  \label{eq:3}
(-1)^k q^{k^2} T_k(t, q^{-1}) =
\sum_{i=0}^k q^{\binom{i}2} (q;q^2)_{k-i} t^{i} \qbinom{2k-i}{i}. 
\end{equation}

In \eqref{eq:Tk} if we replace $j\mapsto k-j$, interchange the $i$-sum and the
$j$-sum, and replace $j\mapsto j+i$, then the $j$-sum is summable by the
$q$-binomial theorem. Thus the left hand side of \eqref{eq:3} is equal to 
\begin{align*}
&\sum_{i=0}^{\flr{k/2}} t^{2i}q^{2i^2-i}(q^{2i+1};q^2)_{k-2i} \Qbinom{k-i}{i}{q^2}
+ \sum_{i=0}^{\flr{k/2}} 
t^{2i+1}q^{2i^2+i}(q^{2i+3};q^2)_{k-2i-1} \Qbinom{k-i-1}{i}{q^2}\\
&= \sum_{i=0}^{\flr{k/2}} t^{2i} q^{2i^2-i} (q;q^2)_{k-2i} \qbinom{2k-2i}{2i}
+ \sum_{i=0}^{\flr{k/2}} t^{2i+1} q^{2i^2+i} (q;q^2)_{k-2i-1} \qbinom{2k-2i-1}{2i+1},
\end{align*}
which is equal to the right hand side of \eqref{eq:3}. 

Zeng \cite{Zeng_PC} showed that
\begin{equation}
  \label{eq:zeng}
E_n(t,q) = t^{-n} 
\sum_{m=0}^n \sum_{i=0}^m (-1)^{n-i}
\frac{q^{2m-2in+i^2-n-i} [2m]_{t,q}! [2i+1]_{t,q}^{2n}}
{[2i]_q!! [2m-2i]_q!! \prod_{k=0,k\ne i}^m [2k+2i+2]_{t^2,q}},
\end{equation}
where $[2m]_{t,q}! = \prod_{i=1}^{2m} [i]_{t,q}$ and
$[2i]_q!! = \prod_{k=1}^i [2k]_q$. 

It is straightforward to check that Zeng's formula is equivalent to
Theorem~\ref{thm:AWWc=-a,d=-c} when $a=\sqrt{-qt}$ and $c=0$.


\end{document}